\documentclass[10pt]{amsart}
\usepackage[]{amsmath, amsthm, amsfonts,  scalerel, amssymb}
\usepackage{graphicx}
\usepackage[all]{xy}
 \usepackage{xcolor}
 \usepackage{tikz-cd}
\usepackage{adjustbox}
\usepackage{hyperref}
\usepackage[
backend=biber,
style=alphabetic,
maxnames=6
]{biblatex}

\usepackage[left=3cm,right=3cm,top=3cm,bottom=3cm]{geometry}
\usepackage{cleveref}
\usepackage[left=3cm,right=3cm,top=3cm,bottom=3cm]{geometry}
\usepackage{cleveref}

\newcommand {\C} {{\mathbb C}}
\newcommand {\R} {{\mathbb R}}
\newcommand {\Z} {{\mathbb Z}}
\newcommand {\Q} {{\mathbb Q}}
\newcommand {\cF} {{\mathcal F}}
\newcommand {\tcF} {\tilde{{\mathcal F}}}

\newcommand {\cH} {{\mathcal H}}
\newcommand {\cG} {{\mathcal G}}

\newcommand {\cE} {{\mathcal E}}

\newcommand {\cM} {{\mathcal M}}

\newcommand {\cO} {{\mathcal O}}
\newcommand {\D} {\mathbb{D}}

\newcommand {\dB}{\underline{\Omega}_X}

\newcommand{\tX}{\tilde{X}}

\newcommand{\tH}{\tilde{H}}

\newcommand{\cD}{\mathcal{D}}

\newcommand{\cX}{\mathcal{X}}
\newcommand{\bR}{\textbf{R}}

\newcommand{\Om}{\underline{\Omega}_{X}^{0}}
\newcommand{\Omp}{\underline{\Omega}_{X}^{p}}

\DeclareMathOperator{\im}{im}

\DeclareMathOperator{\Spec}{Spec}

\DeclareMathOperator\supp{supp}
\DeclareMathOperator\depth{depth}
\DeclareMathOperator\lcd{lcd}

\newcommand{\bT}{\begin{tikzcd}}
\newcommand{\eT}{\end{tikzcd}}

\newtheorem{thm}[subsection]{Theorem}
\newtheorem{cor}[subsection]{Corollary}
\newtheorem{lemma}[subsection]{Lemma}
\newtheorem{prop}[subsection]{Proposition}
\newtheorem{defn}[subsection]{Definition}
\newtheorem{rmk}[subsection]{Remark}
\newtheorem{ex}[subsection]{Example}

\newtheorem{conj}[subsection]{Conjecture}

\newtheorem{nota}[subsection]{Notation}

\begin{document}
\author{Donu Arapura}
\thanks{First author partially supported by a grant from the Simons Foundation}
\address{
 Department of Mathematics\\
  Purdue University\\
  West Lafayette, IN 47907\\
  U.S.A.}
  \email{arapura@purdue.edu}

\date{\today}
\author{Scott Hiatt }
 \address{
 Department of Mathematics\\
  University of Wisconsin-Madison\\
  Madison, WI 53705\\
  U.S.A.}
  \email{shiatt@wisc.edu}

 \title{Differential Forms and Hodge Structures on Singular Varieties}

\maketitle

\begin{abstract}
    We compare a few notions of differential forms on singular complex algebraic varieties and relate them to the outermost associated graded spaces of the Hodge filtration of ordinary and intersection cohomology.  In particular,  we introduce and study singularities, which we call quasi-rational, that are normal and such that for all p, the zeroth cohomology sheaf of the complex of Du Bois p-forms is isomorphic to the direct image of p-forms from a desingularization.  We show that an isolated singularity is rational if and only if it is quasi-rational,
    Du Bois, and specific Hodge numbers of the local mixed Hodge structures vanish.
   
\end{abstract}

\tableofcontents

\section*{Introduction}

In this paper, we work over $\C$. We will consider several notions of (sheaves of) differential forms on a singular algebraic variety $X$:
K\"ahler differentials $\Omega_X^p$,
 the direct image $\tilde \Omega_{X}^p = f_*\Omega_{\tilde X}^p$
 for some (or any) resolution of singularities $f:\tilde X\to X$, and   the Du Bois complex  $\underline{\Omega}_X^p$ and more.
 The first is the simplest but most badly behaved; the quotient $\Omega_X^p/\text{torsion}$ is a bit better. Pull back of forms induces an injective map
 $$\Omega_X^p/\text{torsion}\to \tilde \Omega_X^p.$$
  Even if one does not start out caring about $\underline{\Omega}_X^p$, it enters
 naturally at this point because, as we will see, the above map factors through the zeroth cohomology sheaf $\cH^{0}(\underline{\Omega}_{X}^{p})$. If we take $f: \tX \rightarrow X$ to be a (or any) strong log resolution with exceptional divisor $E$ with simple normal crossing support, we can also consider $\tilde{\Omega}^{p}_{X}(\log )= f_*\Omega_{\tilde X}^p(\log E).$ Finally, if we set $U$ to denote the nonsingular set of $X$ with inclusion map $j: U \hookrightarrow X$, we also have the sheaf $\Omega^{[p]}_{X}:=j_{*}\Omega^{p}_{U}$. With all this notation, we always have the following maps
 $$\Omega_X^p/\text{torsion}\to \cH^{0}(\underline{\Omega}_{X}^{p}) \to \tilde \Omega_X^p \to  \tilde \Omega_X^p(\log ) \to \Omega^{[p]}_{X}.$$
In general, these maps are strict inclusion maps. Our main result is that the vanishing of the higher direct images of the structure sheaf of a (or any) resolution of singularities implies the equivalence of the majority of these different notions of differential forms.

 \begin{thm}
     Let $X$ be a normal variety and $f: \tX \rightarrow X$ a resolution of singularities.      If 
      $R^{i}f_{*}\cO_{\tX} = 0$ for $1 \leq i \leq k$, then the natural maps $\cH^{0}(\underline{\Omega}_{X}^{i}) \rightarrow \tilde \Omega_X^i \to  \tilde \Omega_X^i(\log )  \rightarrow \Omega^{[i]}_{X}$ are isomorphisms for $i \leq k.$
\end{thm}

Equivalently, the previous theorem says that if
      $R^{i}f_{*}\cO_{\tX} = 0$ for $1 \leq i \leq k$, then the  sheaves $\cH^{0}(\underline{\Omega}_{X}^{i}), \tilde \Omega_X^i$, and $\tilde \Omega_X^i(\log )$ are reflexive for $i \leq k$. In particular, if $X$ has rational singularities, then the sheaves $\cH^{0}(\underline{\Omega}_{X}^{p}), \tilde \Omega_X^p$, and $\tilde \Omega_X^p(\log )$ are reflexive for all values of $p$. This correlates with the results given by Kebekus and Schnell \cite{ks} and Greb et. al. \cite{GKKP}. 
      The full version of the statement is given in Theorem \ref{DiffThm}.

      By the last  theorem, if $X$ has rational singularities, then the natural maps  $\cH^{0}(\underline{\Omega}_{X}^{p}) \to \tilde \Omega_X^p$ are all isomorphisms.  We will say that $X$ has  {\em quasi-rational} singularities if it is normal and 
 $\cH^{0}(\underline{\Omega}_{X}^{p}) \rightarrow \tilde \Omega^{p}_{X}$  an isomorphism for all $p$. By what we just stated rational implies quasi-rational,
 but the converse fails even for surfaces. We can characterize isolated rational singularities within this new class of quasi-rational singularities as follows.

\begin{thm}
  Suppose that $X$ is a normal variety with isolated singularities. Then
  $X$ has rational singularities if and only if the following conditions hold:
  \begin{enumerate}
  \item $X$ has quasi-rational singularities. That is, the natural map $\cH^{0}(\underline{\Omega}_{X}^{p}) \to \tilde \Omega_X^p$ is an isomorphism for all $p.$ 
  \item $X$ is Du Bois.
    \item For all $x\in X$, $0<p\le \dim X$, and  $k<p$, the mixed Hodge structures on $H_x^p(X,\C)$ satisfy $Gr_F^0Gr^W_kH_x^p(X,\C)=0$.
  \end{enumerate}
  
\end{thm}

In principle, $\underline{\Omega}_X^i$ can have cohomology sheaves in positive degrees.
We are unable to say much about this for general $i$, but for $i=0$, we can say something
quite precise. 
 
 \begin{thm}\label{VanishingThmintro}
Let $X$ be a normal variety and $f: \tX \rightarrow X$ be a resolution of singularities. If $R^{i}f_{*} \cO_{\tX} = 0$ for $1 \leq i \leq k$, then 
\begin{enumerate}
    \item $\cH^{i}(\underline{\Omega}_{X}^{0}) = 0$ for $1 \leq i \leq k$;\\

    \item the natural map $\cH^{k+1}(\underline{\Omega}_{X}^{0}) \rightarrow R^{k+1}f_{*}\cO_{\tX}$  is injective.
\end{enumerate}
\end{thm}

If $X$ is projective, then the cohomology groups of $\underline{\Omega}_X^p$  gives the associated graded of Deligne's mixed Hodge structure on $H^*(X,\C)$
with respect to the Hodge filtration. So by combining both theorems, we obtain the following corollary.

\begin{cor}
    Suppose that $X$ is a normal projective variety, and let $f: \tX \rightarrow X$ be a resolution of singularities. If $R^{i}f_{*} \cO_{\tX} = 0$ for $1 \leq i \leq k$, then the natural maps
    $$Gr^{0}_{F}H^{i}(X, \C) = H^{i}(X, \underline{\Omega}^{0}_{X}) \rightarrow  H^{i}(\tX, \cO_{\tX})= Gr^{0}_{F}H^{i}(\tX, \C)$$
    $$Gr^{i}_{F}H^{i}(X, \C) =H^{0}(X, \underline{\Omega}^{i}_{X}) \rightarrow H^{0}(\tX, \Omega_{\tX}^{i})= Gr^{i}_{F}H^{i}(\tX, \C)$$
    are isomorphisms for $0 \leq i \leq k$ and injective $i = k+1.$ 
\end{cor}

In fact, the above statement will be deduced from a stronger result obtained earlier in the paper. This
result also compares with Saito's mixed Hodge structure on intersection cohomology.

\begin{thm}\label{HodgeThmIntro}
    Suppose that $X$ is a normal quasi-projective variety, and let $f: \tX \rightarrow X$ be a resolution of singularities. 
  There are natural maps 
    $$Gr^{0}_{F}H^{i}_{c}(X, \C) \rightarrow Gr^{0}_{F}IH^{i}_{c}(X, \C) \xrightarrow{\sim}  Gr^{0}_{F}H^{i}_{c}(\tX, \C) $$
$$Gr^{i}_{F}H^{i}_{c}(X, \C) \rightarrow Gr^{i}_{F}IH^{i}_{c}(X, \C) \xrightarrow{\sim}  Gr^{i}_{F}H^{i}_{c}(\tX, \C)$$
with isomorphisms as indicated for all $i$.
    If $R^{i}f_{*} \cO_{\tX} = 0$ for $1 \leq i \leq k$, then all the  maps 
    are isomorphisms for $0 \leq i \leq k$ and injective for $i = k+1$.
\end{thm}

An interesting application of this last theorem is a partial answer to a conjecture given by Musta\c{t}\u{a} and Popa \cite{MuPo}. 
\begin{conj}
     If $\depth(\cO_{X}) \geq i +2$, then $\depth(\underline{\Omega}^{i}_{X}) \geq 2$.
 \end{conj}
We can prove the conjecture if $X$ has rational singularities outside a finite set. Otherwise, we get the following lower bound on the depth of $\cO_X$.

 \begin{thm}
      If $\tilde{S}$ is  the smallest closed subset of $X$ such that $X \backslash \tilde{S}$ has rational singularities, and $n_{\tilde{S}} = \dim \tilde{S}$, then
    $$\depth(\underline{\Omega}^{i}_{X}) \geq 2 \quad \text{for $i  + n_{\tilde{S}}+2 \leq \depth(\cO_{X}).$}$$
 \end{thm}
 This theorem is actually a corollary to a stronger theorem.

We would like to thank the referee for careful reading, and many helpful comments.

\section{Vanishing theorems}
For the remainder of this paper, $X$ will be an irreducible complex variety of dimension $n$, where a variety will always mean a reduced, separated scheme of finite type over $\C$. In this section,
$Z \subseteq X$ will be a closed sub-variety of $X$ with  dimension $n_{Z}$. We will also denote $U \subseteq X$ as the complement of $Z$  with natural map $j: U \hookrightarrow X$. 

    \subsection*{A vanishing theorem for local cohomology sheaves}

    We begin by generalizing the results given in \cite[\S 5]{hiatt1}, which proved the following results when $X$ is projective and in special cases for $Z.$

    \begin{nota}
In this paper, $\omega^{\bullet}_{X} \in D^{b}_{coh}(\cO_{X})$ will denote the dualizing complex of $X$. Note that when $X$ is smooth, $\omega^{\bullet}_{X} \simeq \Omega^{\dim X}_{X}[\dim X]$. When $X$ is smooth, we define $\omega_{X}:= \Omega^{\dim X}_{X}[\dim X]. $
\end{nota}

    \begin{nota}\label{nota1.8}
Given a complex of sheaves $\cF$, we have the truncations
$$\sigma_{>n}\cF: \cdots \rightarrow 0 \rightarrow \cF^{n+1} \rightarrow \cF^{n+2} \rightarrow \cdots$$
$$\sigma_{\leq n}\cF: \cdots \rightarrow \cF^{n-1} \rightarrow \cF^{n} \rightarrow 0 \rightarrow \cdots$$
$$\tau^{>n}\cF: \cdots \rightarrow 0 \rightarrow \im d^{n} \rightarrow \cF^{n+1} \rightarrow \cdots$$
$$\tau^{\leq n}\cF: \cdots \rightarrow \cF^{n-1} \rightarrow \ker d^{n} \rightarrow 0 \rightarrow \cdots.$$
\end{nota}
    
\begin{thm}\label{MainThm}
     Suppose $\cF^{\bullet}\in D^{b}_{coh}(\cO_{X})$ is quasi-isomorphic to a complex $\tcF^{\bullet}$ with the following properties:
    \begin{enumerate}
        \item $\sigma_{\leq -n-1}(\tcF^{\bullet}\vert_{U}) = 0$;
        \item There exists $\ell \in \{0,1, \cdots ,n\}$ such that $\sigma_{>-n + \ell}(\tcF^{\bullet} \vert_{U}) =  0$;
         \item For all $j \in \Z$, if $\tcF^{j}\vert_{U} \neq 0$, then $\tcF^{j}\vert_{U}$  is a maximal Cohen-Macaulay sheaf on $U$;
        \item There exists an integer $q \geq -n +\ell +1$ such that $\cH^{i}(\tcF) = 0$ for all $i \geq q$.

    \end{enumerate}
    Then $\cH^{p}(\bR \underline{\Gamma_{Z}}\bR \cH om_{X}(\cF^{\bullet}, \omega_{X}^{\bullet})) = 0 \quad \text{for $p \leq -n_{Z} - q$}$.
\end{thm}

\begin{lemma}\label{AssLemma}
    Let $R$ be a noetherian ring and $M$ be a finitely generated $R$ module. If $\Gamma_{PR_{P}}(M_{P}) = 0$ for all $P \in Supp (M)$, then $M = 0.$
\end{lemma}

\begin{proof}
    By \cite[pg. 100]{Eisenbud} $Ass(\Gamma_{P}(M)) = \{Q \in Ass(M)\vert$ $P \subset Q\}.$ If we localize at $P$, then $Ass(\Gamma_{PR_{P}}(M_{P})) = \{PR_{P}\}$ if and only if $P \in Ass(M).$ If  $\Gamma_{PR_{P}}(M_{P}) = 0$ for all $P \in Supp(M)$, then $Ass(M) = \emptyset$. If $Ass(M) = \emptyset$, then $M = 0.$
\end{proof}

\begin{lemma}
    With the same assumptions as the theorem, $\cH^{p}(\bR \underline{ \Gamma_{Z}}\bR \cH om_{X}(\cF^{\bullet}, \omega_{X}^{\bullet}))$ is coherent for $p \leq -n_{Z} - q.$
\end{lemma}

\begin{proof}
    We follow the proof of \cite[Thm. 5.1]{hiatt1}. First, we may assume $\cF = \tcF.$ Consider the exact triangle
    $$ \begin{tikzcd} \bR \underline{\Gamma_{Z}}\bR \cH om_{X}(\cF^{\bullet}, \omega_{X}^{\bullet})  \arrow[r] & \bR \cH om_{X}(\cF^{\bullet}, \omega_{X}^{\bullet}) \arrow[r] & \bR j_{*} j^{*}\bR \cH om_{X}(\cF^{\bullet}, \omega_{X}^{\bullet}) \arrow[r, "+1"] & \hfill. \end{tikzcd}$$
    Since $\cF \in D^{b}_{coh}(\cO_{X})$, we have $\bR \cH om_{X}(\cF^{\bullet}, \omega_{X}^{\bullet}) \in D^{b}_{coh}(\cO_{X}).$ So, by the exact triangle above, it suffices to show $\cH^{p}(\bR j_{*} j^{*}\bR \cH om_{X}(\cF^{\bullet}, \omega_{X}^{\bullet}))$ is coherent for $p \leq -n_{Z} - q -1.$ By our assumptions, the complex $\cF^{\bullet}\vert_{U}$ is given by
$$\cdots \rightarrow 0 \rightarrow \cF^{-n}\vert_{U} \rightarrow \cF^{-n+1}\vert_{U} \rightarrow \cdots \rightarrow \cF^{ -n+\ell}\vert_{U} \rightarrow 0 \rightarrow \cdots.$$
 The complex $ j^{*} \bR \cH om_{X}(\cF^{\bullet}, \omega^{\bullet}_{X})$ is given by
$$\cdots \rightarrow 0 \rightarrow (\cF^{-n+\ell}\vert_{U})^{\vee}  \rightarrow (\cF^{ -n +\ell -1}\vert_{U})^{\vee} \rightarrow \cdots \rightarrow (\cF^{-n}\vert_{U})^{\vee}  \rightarrow 0 \rightarrow \cdots,$$
where $(\cF^{j}\vert_{U})^{\vee} = \cH^{-n}(\bR \cH om_{U}(\cF^{j}\vert_{U}, \omega^{\bullet}_{U}))$. Note that the sheaf $(\cF^{j}\vert_{U})^{\vee}$ is in the $-(j +n)^{th}$ spot for $j \in \{-n, -n +\ell\}$. To simplify the notation, we will let $\cG:= j^{*} \bR \cH om_{X}(\cF^{\bullet}, \omega^{\bullet}_{X}) $, and we will show $R^{p}j_{*}\cG$ is coherent for $p \leq - n_{Z} - q -1.$ By \cite[Lemma I.7.2]{Hart1}, there is an exact sequence of complexes given by
$$ 0 \rightarrow \sigma_{> -\ell}\cG \rightarrow \cG \rightarrow \cG^{-\ell}[\ell] \rightarrow 0.$$
By applying $\bR j_{\ast}$ to the exact sequence above, there is an exact triangle
$$\begin{tikzcd}
 \bR j_{\ast}\sigma_{> -\ell}\cG \arrow[r] & \bR j_{\ast}\cG \arrow[r]& \bR j_{\ast}\cG^{-\ell}[\ell] \arrow[r, "+1"] & \hfill. 
 \end{tikzcd}$$
By induction on $\ell$, it suffices to show $R^{p}j_{\ast}\cG^{-\ell}[\ell]$ is coherent for $p \leq - n_{Z} -q -1.$ Which is equivalent to showing $R^{p}j_{\ast}\cG^{-\ell}$ is coherent for $p \leq  \ell- n_{Z}-q  - 1.$ Note that we have inequalities
$$p \leq  \ell- n_{Z}-q  - 1 \leq n -n_{Z} -2.$$ 
Because the sheaf $\cG^{-\ell} $ is a maximal Cohen-Macaulay sheaf $U$, the sheaf $R^{p}j_{\ast}\cG^{-\ell}$ is coherent for $p \leq n - n_{Z} -2$  \cite{Siu} (see also \cite[VI. vi]{Hart3}).  Hence $R^{p}j_{\ast}\cG^{-\ell}$ is coherent for $p \leq \ell -n_{Z} -q -1$.

\end{proof}

\begin{proof}[Proof of Theorem \ref{MainThm}.] We may assume $X$ is affine because the problem is local. Let $X = \Spec(R)$, $I$ the ideal for $Z$, and $P$ be any prime ideal such that $I \subset P$. We will denote $n(P)$ for $ \dim R/P.$ There is a spectral sequence given by
    $$ \cH^{a}(\bR \underline{\Gamma_{P}}( \cH^{b}(\bR \underline{\Gamma_{I}}\bR \cH om_{R}(\cF^{\bullet}, \omega_{R}^{\bullet}))) \Rightarrow \cH^{a +b}(\bR \underline{\Gamma_{P}} \bR \underline{\Gamma_{I}}\bR \cH om_{R}(\cF^{\bullet}, \omega_{R}^{\bullet})) \simeq \cH^{a +b}(\bR \underline{\Gamma_{P}}\bR \cH om_{R}(\cF^{\bullet}, \omega_{R}^{\bullet})).$$
If we consider the local ring $(A,m) = (R_{P},PR_{p})$, for any $M^{\bullet}\in D^{b}_{coh}(A)$, by local duality \cite[Cor. 6.3]{Hart1}, there is an isomorphism
$$\bR \underline{\Gamma_{m}}(M^{\bullet}) \simeq \cH om \bigg(\bR  \cH om_{A}(M^{\bullet
},\omega^{\bullet}_{A}),\kappa(P) \bigg),$$
where $\kappa(P)$ is the injective hull of $A/m.$
If we let $M^{\bullet} = \bR \cH om_{R}(\cF^{\bullet}, \omega_{R}^{\bullet})_{P}$, there is a  quasi-isomorphism
$$(\bR \underline{\Gamma_{P}}\bR \cH om_{R}(\cF^{\bullet}, \omega_{R}^{\bullet}))_{P} \simeq \cH om_{R_{P}}(\cF^{\bullet}_{P}, \kappa(P))[n(P)].$$
Hence
$$\cH^{a +b}(\bR \underline{\Gamma_{P}}\bR \cH om_{R}(\cF^{\bullet}, \omega_{R}^{\bullet}))_{P} \cong  \cH om_{R_{P}}(\cH^{-a -b - n(P)}(\cF^{\bullet}_{P}), \kappa(P)).$$
Suppose $\{b \in \Z \vert$ $b \leq -n_{Z} - q$ and $\cH^{b}(\bR \underline{\Gamma_{I}}\bR \cH om_{R}(\cF^{\bullet}, \omega_{R}^{\bullet})) \neq 0\} \neq \emptyset$ and let 
\begin{center}
    $b_{m} = \min \{b \in \Z \vert$ $b \leq -n_{Z} - q$ and $\cH^{b}(\bR \underline{\Gamma_{I}}\bR \cH om_{R}(\cF^{\bullet}, \omega_{R}^{\bullet})) \neq 0\}.$ 
\end{center}
Then, by the spectral sequence above,
$$\underline{\Gamma_{PR_{P}}}(\cH^{b_{m}}(\bR \underline{\Gamma_{I}}\bR \cH om_{R}(\cF^{\bullet}, \omega_{R}^{\bullet}))_{P}) \cong \cH^{b_{m}}(\bR \underline{\Gamma_{P}}\bR \cH om_{R}(\cF^{\bullet}, \omega_{R}^{\bullet}))_{P} \cong  \cH om_{R_{P}}(\cH^{-b_{m} - n(P)}(\cF^{\bullet}_{P}), \kappa(P)).$$
Notice that $-b_{m} - n(P) \geq q + n_{Z} - n(P) \geq q.$ By the fourth assumption of the theorem,
$$\underline{\Gamma_{PR_{P}}}(\cH^{b_{m}}(\bR \underline{\Gamma_{I}}\bR \cH om_{R}(\cF^{\bullet}, \omega_{R}^{\bullet}))_{P}) \cong  \cH om_{R_{P}}(\cH^{-b_{m} - n(P)}(\cF^{\bullet}_{P}), \kappa(P)) =0.$$ 
This isomorphism holds for all $P \in \Spec (R)$ such that $I \subset P$. By Lemma \ref{AssLemma}, we must have $\cH^{b_{m}}(\bR \underline{\Gamma_{I}}\bR \cH om_{R}(\cF^{\bullet}, \omega_{R}^{\bullet})) = 0$. Therefore $\cH^{b}(\bR \underline{\Gamma_{I}}\bR \cH om_{R}(\cF^{\bullet}, \omega_{R}^{\bullet})) = 0$ for $b \leq -n_{Z} - q.$

\end{proof}

\begin{cor}\label{MainCor}
With the same assumptions as Theorem \ref{MainThm}, the natural map
$$ \cH^{p}(\bR \cH om_{X}(\cF^{\bullet}, \omega^{\bullet}_{X}))\rightarrow \cH^{p}(\bR j_{\ast}\bR \cH om_{U}(\cF^{\bullet}\vert_{U}, \omega^{\bullet}_{U}))$$
is an isomorphism for $p \leq -n_{Z} -q -1$ and injective for $p = -n_{Z} - q$.
\end{cor}

\begin{proof}
Apply the previous theorem to the exact triangle
 $$ \begin{tikzcd}\bR \underline{\Gamma_{Z}}\bR \cH om_{X}(\cF^{\bullet}, \omega_{X}^{\bullet})  \arrow[r] & \bR \cH om_{X}(\cF^{\bullet}, \omega_{X}^{\bullet}) \arrow[r] & \bR j_{*} j^{*}\bR \cH om_{X}(\cF^{\bullet}, \omega_{X}^{\bullet}) \arrow[r, "+1"] & \hfill. \end{tikzcd}$$
\end{proof}

\begin{cor}\label{LogCor}
     Let $S$ be the singular locus of $X$ with dimension $n_{S}$. Let $f:\tX \rightarrow X$ be a strong log resolution. Then 
     $$\cH^{p}(\bR \underline{\Gamma_{S}} \bR f_{*}\Omega^{q}_{\tX}(\log E)) = 0 \quad \text{for $p +q \leq n -n_{S}  -1.$}$$

\end{cor}

\begin{proof}
      If we let $U = X \backslash S$, then we may apply Theorem \ref{MainThm} to the complex $\bR f_{*} \Omega^{n-p}_{\tX}(\log E)\in D^{b}_{coh}(\cO_{X})$ with $\ell = n$. Note that $\bR f_{*} \Omega^{n-p}_{\tX}(\log E) \vert_{U} \simeq \Omega^{n-p}_{U}$, and  $\Omega^{n-p}_{U}$ is a maximal Cohen-Macaulay module on $U$ because $U$ is smooth. Also, $R^{i}f_{*}\Omega^{n-q}_{\tX}(\log E)(-E) = 0$ for $i \geq q +1$ by Steenbrink's vanishing theorem \cite{steenbrink}. So, we set $q$ in Theorem \ref{MainThm} to be $q+1.$  We have the quasi-isomorphisms
    $$\cH^{p-n}(\bR \underline{\Gamma_{S}}\bR \cH om_{X}(\bR f_{*} \Omega^{n-q}_{\tX}(\log E)(-E), \omega^{\bullet}_{X})) \cong \cH^{p-n} (\bR \underline{\Gamma_{S}} \bR f_{*} \Omega^{q}_{\tX}(\log E)[n]) \cong  \cH^{p} (\bR \underline{ \Gamma_{S}} \bR f_{*} \Omega^{q}_{\tX}(\log E)),$$
    By Theorem \ref{MainThm},
    $$\cH^{p}(\bR \underline{\Gamma_{S}} \bR f_{*}\Omega^{q}_{\tX}(\log E)) = 0 \quad \text{for $p -n \leq -q -n_{S}  -1.$}$$
\end{proof}

\begin{defn}
Let $Y$ be a locally Noetherian scheme. If $(R,m)$ is the local ring for a closed point $y \in Y$ and $M$ is an element of the bounded derived category of $R$-modules, we define 
$$\depth(M) := \min\{i \hspace{.03in} | \hspace{.03in} \cH^{-i}(\bR H om_{R}(M, \omega_{R}^{\bullet}))\neq 0\} = \min \{i \hspace{.03in} | \hspace{.03in} \cH^{i}(\bR \underline{\Gamma_{m}}(M)) \neq 0 \}.$$ 
If $F$ is a coherent sheaf on $Y$, the depth of $F$ is given by
$$\depth(F):= \min_{y \in Y}(\depth(F_{y})),$$
where $y \in Y$ is any closed point. 

If $\Delta$ a closed subset of $Y$, then we define $$\depth_{\Delta}(F) := \displaystyle \inf_{z \in \Delta}(\depth F_{z}),$$
where $z\in \Delta$ is not necessarily closed. 
\end{defn}

\begin{lemma}\label{depthlemma}
    With the same notation as the previous definition, if the dimension of $\Delta$ is $n_{\Delta},$ then we always have
    $$\depth_{\Delta}(F) \geq \depth (F) - n_{\Delta}.$$
\end{lemma}

\begin{proof}
    This follows from a well-known inequality in commutative algebra. If $R$ is a noetherian local ring and $M$ is a coherent module, then for any $p \in \Spec(R)$,
    $$\depth(M_{p}) + \dim(R/p) \geq \depth(M)$$
\cite[\href{https://stacks.math.columbia.edu/tag/0FCC}{Tag 0FCC}]{stacks-project}.
\end{proof}

\begin{cor}\label{rat.rmk}
 Let $f: \tX \rightarrow X$ be a resolution of singularities. If $\tilde{S}$ is the the smallest closed subset of $X$ such that $X \backslash \tilde{S}$ has rational singularities, then
    $$R^{i}f_{\ast}\cO_{\tX} = 0 \quad \text{for $1 \leq i \leq \depth_{\tilde{S}}(\cO_{X}) -2$}.$$
    \end{cor}
\begin{proof}
Let $\tilde{U} = X \backslash \tilde{S}$ with inclusion map $\tilde{j}: \tilde{U} \hookrightarrow X$. The complex $\bR f_{*}\cO_{\tX}$ is independent of the resolution of singularities. So, we may assume $f: \tX \rightarrow X$ is a strong log resolution. Furthermore, we may choose a resolution of singularities that is functorial with respect to smooth morphisms \cite[Thm. 36]{kollar3}. So, if $\tilde{V} = f^{-1}(\tilde{U})$ with induced map $g: \tilde{V} \rightarrow \tilde{U}$, then we have $\bR g_{*} \cO_{\tilde{V}} \simeq \cO_{\tilde{U}}.$ 

By the Grauert-Riemenschneider vanshing theorem, if $n_{\tilde{S}} = \dim \tilde{S},$ then $\cH^{i}(\bR \underline{\Gamma}_{\tilde{S}}(\bR f_{*} \cO_{\tX})) = 0$ for $i \leq n - n_{\tilde{S}} -1.$ Therefore, the natural map $R^{i}f_{*}\cO_{\tX} \rightarrow R^{i}\tilde{j}_{*}\cO_{\tilde{U}}$ is an isomorphism for $i \leq n - n_{\tilde{S}} -2$. Note that $\depth_{\tilde{S}}(\cO_{X}) -2 \leq n - n_{\tilde{S}} -2$ and equality holds when $X$ is Cohen-Macaulay. For $1 \leq i \leq \depth_{\tilde{S}}(\cO_{X}) -2$, $R^{i}\tilde{j}_{*}\cO_{U} \cong \cH^{i+1}_{\tilde{S}}(\cO_{X}) = 0$ \cite[Thm. 3.8]{Hart2}.
\end{proof}

\begin{rmk}\label{rat.rmk2}
   If $\tilde{S}$ is the the smallest closed subset of $X$ such that $X \backslash \tilde{S}$ has rational singularities, and $n_{\tilde{S}}$ is the dimension of $\tilde{S}$, then, by Lemma \ref{depthlemma} and the previous corollary, we obtain
     $$R^{i}f_{\ast}\cO_{\tX} = 0 \quad \text{for $1 \leq i \leq \depth(\cO_{X}) -n_{\tilde{S}} -2$}.$$
      A proof of this statement was also given by Popa, Shen, and Vo \cite[Corollary 11.3]{PopaShenVo} (the case when $k = 0).$
\end{rmk}

\begin{rmk}
    Let $f: \tX \rightarrow X$ is a resolutions of singularities. Note that if $X$ is normal and $R^{i}f_{*}\cO_{\tX} = 0$ for $1\leq i \leq k$, then $\depth(\cO_{X}) \geq k$.  This can be easily seen by using the spectral sequence
    $$E^{i,j}_{2} = \cE xt^{i}_{X}(R^{-j}f_{*}\cO_{\tX}, \omega^{\bullet}_{X}) \Rightarrow \cE xt^{i+j}_{X}(\bR f_{*}\cO_{\tX}, \omega^{\bullet}_{X})$$
    and the fact that $R^{i+j}f_{*}\omega_{\tX} \cong \cE xt^{i+j}_{X}(\bR f_{*}\cO_{\tX}, \omega^{\bullet}_{X}) = 0$ unless $i + j =-n.$ We thank the referee for pointing this out.
\end{rmk}

\subsection*{A vanishing theorem for the Du Bois complex}
Recall that to any complex algebraic variety $X$, we can associate a filtered complex $(\underline{\Omega}^{\bullet}_{X}, F)$, or more precisely, an object in the filtered derived category, c.f. \cite[\S 1.1]{dubois}. We define $ \underline{\Omega}^{p}_{X}=Gr^{p}_{F}\underline{\Omega}^{\bullet}_{X}[p]$, which is an object in $D^{b}_{coh}(\cO_{X})$ \cite{dubois}.  A rather different and enlightening construction of the last object can be found in \cite{huber}. These have the following properties:
\begin{enumerate}
\item If we forget the filtration, then there is an isomorphism $\underline{\Omega}^{\bullet}_{X} \simeq \C_{X}$ in the usual derived category.
\item If $\Omega^{\bullet}_{X}$ is the usual de Rham complex with the ``stupid" filtration, there exists a natural map of filtered complexes
$$(\Omega^{\bullet}_{X},F) \rightarrow (\underline{\Omega}^{\bullet}_{X},F)$$
If $X$ is smooth, this map is an isomorphism. In particular, we have a morphism
$$\cO_{X} \rightarrow \underline{\Omega}^{0}_{X}.$$
\item If $\pi_\bullet: X_\bullet \to X$ is a simplicial resolution (i.e., an augmented smooth semi-simplicial scheme satisfying cohomological descent), then by construction
$$\underline{\Omega}_X^p \cong \R \pi_{\bullet} \Omega_{X_\bullet}^p$$

\end{enumerate}

One useful and well-known consequence of (3) is:

\begin{lemma}\label{lemma:dBtorsionfree}
 If $X$ is a variety, then $\cH^0(\underline{\Omega}_X^p)$ is torsion free.
 If $X$ is normal, then $\cO_X=\cH^0(\underline{\Omega}_X^0)$.
\end{lemma}

\begin{proof}
 We can choose a simplicial resolution such that $\pi_0:X_0\to X$ is a resolution of singularities, then the above formula implies
 $$\cH^0(\underline{\Omega}_X^p)\subseteq \pi_{0*} \Omega_{X_0}^p$$
 This implies torsion freeness. The natural map $\cO_X\to \cH^0(\underline{\Omega}_X^0)$ is an isomorphism on the nonsingular locus $X-S$.
 If $X$ is normal, this must extend to an isomorphism everywhere because $\depth_{S}(\cO_{X})\ge 2$.
\end{proof}

If $\underline{\omega}_{X}^{\bullet} = \bR \cH om_{X}(\underline{\Omega}^{0}_{X}, \omega_{X}^{\bullet})$, then it was shown by Kov\'acs and Schwede \cite{KovacsSch} that the dual map $\cH^{i}(\underline{\omega}_{X}^{\bullet}) \rightarrow \cH^{i}(\omega_{X}^{\bullet})$ is injective. This was also shown by Musta\c{t}\u{a} and Popa \cite{MuPo}. Equivalently, if $p \in X$, then the natural map
    $$\cH^{i}(\bR \underline{\Gamma_{p}}(\cO_{X,p}))\rightarrow  \cH^{i}(\bR \underline{\Gamma_{p}}( (\underline{\Omega}_{X}^{0})_{p}))$$
    is surjective. We use this fact to prove the following lemma.

\begin{lemma}
    Let $X$ be a normal variety and $f: \tX \rightarrow X$ be a resolution of singularities. If $p \in X$ (not necessarily closed) and $R^{i}f_{*} \cO_{\tX} = 0$ for $1 \leq i \leq k$, then
     $$\cH^{i}(\bR \underline{\Gamma_{p}}(\cO_{X,p}))\cong  \cH^{i}(\bR \underline{\Gamma_{p}}( (\underline{\Omega}_{X}^{0})_{p})) \quad \text{for $0 \leq i \leq k$.}$$
\end{lemma}

\begin{proof}
     Notice that if $X$ is normal and $R^{i}f_{*}(\cO_{\tX}) = 0$ for $1 \leq i \leq k$, then 
    $$\cH^{i}(\bR \underline{\Gamma_{p}}(\cO_{X,p}))\cong  \cH^{i}(\bR \underline{\Gamma_{p}}( (\bR f_{*}\cO_{X})_{p})) \quad \text{for $0 \leq i \leq k$.}$$
    Since there is a natural factorization 
    $$\cH^{i}(\bR \underline{\Gamma_{p}}(\cO_{X,p}))\rightarrow  \cH^{i}(\bR \underline{\Gamma_{p}}( (\underline{\Omega}_{X}^{0})_{p})) \rightarrow \cH^{i}(\bR \underline{\Gamma_{p}}( (\bR f_{*}\cO_{X})_{p})), $$
    we have 
    $$\cH^{i}(\bR \underline{\Gamma_{p}}(\cO_{X,p})) \cong  \cH^{i}(\bR \underline{\Gamma_{p}}( (\underline{\Omega}_{X}^{0})_{p}))  \quad \text{for $0 \leq i \leq k$.}$$

\end{proof}

\begin{thm}\label{VanishingThm}
Let $X$ be a normal variety and $f: \tX \rightarrow X$ be a resolution of singularities. If $R^{i}f_{*} \cO_{\tX} = 0$ for $1 \leq i \leq k$, then 
\begin{enumerate}
    \item $\cH^{i}(\underline{\Omega}_{X}^{0}) = 0$ for $1 \leq i \leq k$;\\

    \item the natural map $\cH^{k+1}(\underline{\Omega}_{X}^{0}) \rightarrow R^{k+1}f_{*}\cO_{\tX}$  is injective.
\end{enumerate}
\end{thm}

\begin{proof}
    If $p \in X$, there is a morphism of spectral sequences
    $$\begin{tikzcd}
        E^{ji}_{2} = \cH^{j}(\bR \underline{\Gamma_{p}}(\cH^{i}(\Om)_{p})\arrow[r, Rightarrow] \arrow[d] & \cH^{i +j}(\bR \underline{\Gamma_{p}}((\underline{\Omega}_{X}^{0})_{p}))  \arrow[d] \\
       ^{'}E^{ji}_{2}= \cH^{j}(\bR \underline{\Gamma_{p}}(\cH^{i}(\bR f_{*}\cO_{\tX})_{p})  \arrow[r, Rightarrow] &  \cH^{i +j}(\bR \underline{\Gamma_{p}}((\bR f_{*}\cO_{\tX})_{p})).
    \end{tikzcd} $$
For $i = 1$, there is a commutative diagram 
    $$ \begin{tikzcd}
        \underline{\Gamma_{p}}(\cH^{1}(\Om)_{p})\arrow[r, "d^{01}_{2}"] \arrow[d] &   \cH^{2}(\bR \underline{\Gamma_{p}}(\cO_{X,p})) \arrow[d,"id"] \\
        \underline{\Gamma_{p}}((R^{1}f_{*}\cO_{\tX})_{p})  \arrow[r, "^{'}d^{01}_{2}"] &  \cH^{2}(\bR \underline{\Gamma_{p}}(\cO_{X,p})).
    \end{tikzcd} $$
    Since $R^{1}f_{*}\cO_{\tX} = 0$, the map 
    $$\begin{tikzcd}
        \underline{\Gamma_{p}}(\cH^{1}(\Om)_{p})\arrow[r, "d^{01}_{2}"]  & \cH^{2}(\bR \underline{\Gamma_{p}}(\cO_{X,p}))
    \end{tikzcd} $$
    is the zero map. Therefore, $E^{01}_{\infty} \cong \underline{\Gamma_{p}}(\cH^{1}(\Om)_{p})$. But, by the previous lemma, we also have $$E^{10}_{\infty} \cong \cH^{1}(\bR \underline{\Gamma_{p}}(\cO_{X,p})) \cong  \cH^{1}(\bR\underline{ \Gamma_{p}}((\underline{\Omega}_{X}^{0})_{p})). $$
     So, we can conclude that $0 = E^{01}_{\infty} \cong \underline{\Gamma_{p}}(\cH^{1}(\Om)_{p})$. But since $\underline{\Gamma_{p}}(\cH^{1}(\Om)_{p}) = 0$ for all $p \in X$, we can conclude $\cH^{1}(\Om) = 0$ by Lemma \ref{AssLemma}.

      Since we have shown $\cH^{1}(\Om) = 0$, we have $E^{02}_{2} = E^{02}_{3}.$ There is a commutative diagram 
    $$ \begin{tikzcd}
        \underline{\Gamma_{p}}(\cH^{2}(\Om)_{p})\arrow[r, "d^{02}_{3}"] \arrow[d] &   \cH^{3}(\bR \underline{\Gamma_{p}}(\cO_{X,p})) \arrow[d,"id"] \\
        \underline{\Gamma_{p}}((R^{2}f_{*}\cO_{\tX})_{p})  \arrow[r, "^{'}d^{02}_{3}"] &  \cH^{3}(\bR \underline{\Gamma_{p}}(\cO_{X,p})).
    \end{tikzcd} $$
    Since $R^{2}f_{*}\cO_{\tX} = 0$, the map 
    $$\begin{tikzcd}
        \underline{\Gamma_{p}}(\cH^{2}(\Om)_{p})\arrow[r, "d^{01}_{2}"]  & \cH^{3}(\bR \underline{\Gamma_{p}}(\cO_{X,p}))
    \end{tikzcd} $$
    is the zero map. Therefore, $E^{02}_{\infty} \cong \underline{\Gamma_{p}}(\cH^{2}(\Om)_{p})$. But, by the previous lemma, we also have $$E^{20}_{\infty} \cong \cH^{2}(\bR \underline{\Gamma_{p}}(\cO_{X,p})) \cong  \cH^{2}(\bR\underline{ \Gamma_{p}}((\underline{\Omega}_{X}^{0})_{p})). $$
     So, we can conclude that $0 = E^{02}_{\infty} \cong \underline{\Gamma_{p}}(\cH^{2}(\Om)_{p})$. But since $\underline{\Gamma_{p}}(\cH^{2}(\Om)_{p}) = 0$ for all $p \in X$, we can conclude $\cH^{2}(\Om) = 0$ by Lemma \ref{AssLemma}. Continuing in this pattern establishes $(1).$

      To prove $(2)$, let $p \in X$ and notice that 
      $$E^{0,k+1}_{\infty} = \ker \bigg( \underline{\Gamma_{p}}(\cH^{k+1}(\Om)_{p})\xrightarrow{d^{0,k+1}_{k+2}} \cH^{k+2}(\bR \underline{\Gamma_{p}}(\cO_{X,p})) \bigg) $$
      $$E^{k+1,0}_{\infty} \cong \cH^{k+1}(\bR \underline{\Gamma_{p}}(\cO_{X,p})).$$
      Since $\cH^{k+1}(\bR \underline{\Gamma_{p}}(\cO_{X,p}))\rightarrow  \cH^{k+1}(\bR \underline{\Gamma_{p}}( (\underline{\Omega}_{X}^{0})_{p}))$ is surjective, the map must be an isomoprhism because $E^{k+1,0}_{\infty} \cong \cH^{k+1}(\bR \underline{\Gamma_{p}}(\cO_{X,p})).$ Hence, we have 
       $$E^{0,k+1}_{\infty} = \ker \bigg( \underline{\Gamma_{p}}(\cH^{k+1}(\Om)_{p})\xrightarrow{d^{0,k+1}_{k+2}} \cH^{k+2}(\bR \underline{\Gamma_{p}}(\cO_{X,p})) \bigg) = 0. $$
       From the commutative diagram 
      $$ \begin{tikzcd}
        \underline{\Gamma_{p}}(\cH^{k+1}(\Om)_{p})\arrow[r, "d^{0 k+1}_{k+2}"] \arrow[d] &   \cH^{k+2}(\bR \underline{\Gamma_{p}}(\cO_{X,p})) \arrow[d,"id"] \\
        \underline{\Gamma_{p}}((R^{k+1}f_{*}\cO_{\tX})_{p})  \arrow[r, "^{'}d^{0 k+1}_{k+2}"] &  \cH^{k+2}(\bR \underline{\Gamma_{p}}(\cO_{X,p})),
    \end{tikzcd} $$
        there is an inclusion
        $$\ker \bigg( \underline{\Gamma_{p}}(\cH^{k+1}(\Om)_{p}) \rightarrow   \underline{\Gamma_{p}}((R^{k+1}f_{*}\cO_{\tX})_{p}) \bigg) \subseteq \ker \bigg( \underline{\Gamma_{p}}(\cH^{k+1}(\Om)_{p})\xrightarrow{d^{0,k+1}_{k+2}} \cH^{k+2}(\bR \underline{\Gamma_{p}}(\cO_{X,p})) \bigg) = 0. $$
      Since this holds for all $p \in X$, the morphism $\cH^{k+1}(\underline{\Omega}_{X}^{0}) \rightarrow R^{k+1}f_{*}\cO_{\tX}$  is injective.

\end{proof}

\begin{cor}
    If $X$ is a normal variety and  $\tilde{S}$ is the the smallest closed subset of $X$ such that $X \backslash \tilde{S}$ has rational singularities, then 
    $$\cH^{i}(\underline{\Omega}_{X}^{0}) = 0 \quad \text{for $1 \leq i \leq \depth_{\tilde{S}}(\cO_{X}) -2.$}$$
    In particular, if $X$ has rational singularities, then $X$ has Du Bois singularities.
\end{cor}

\begin{proof}
 The first statement follows from the theorem. The second statement follows by combining this with Lemma \ref{lemma:dBtorsionfree}.
\end{proof}

The last statement was originally due to Kov\'acs \cite{kovacs} and Saito \cite{saito5}.

\begin{rmk}
    By Lemma \ref{depthlemma}, if $\tilde{S}$ is the smallest closed subset of $X$ such that $X \backslash \tilde{S}$ has rational singularities, and $n_{\tilde{S}} = \dim \tilde{S}$, then 
    $$\cH^{i}(\underline{\Omega}_{X}^{0}) = 0 \quad \text{for $1 \leq i \leq \depth(\cO_{X})-n_{\tilde{S}} -2.$}$$
    If we set $S_{db}$ to be the smallest closed subset of $X$ such that $X \backslash S_{db}$ has Du Bois singularities, and $n_{S_{db}} = \dim S_{db}$, then it was shown by Popa, Shen, and Vo \cite{PopaShenVo} that we have the following better result
    $$\cH^{i}(\underline{\Omega}_{X}^{0}) = 0 \quad \text{for $1 \leq i \leq \depth(\cO_{X})-n_{S_{db}} -2.$}$$
\end{rmk}

\section{The Hodge filtrations}

In this section, we will work in the analytic category. A variety $X$ will be treated as the analytic space $X^{an}$.
Given an $n$-dimensional algebraic variety $X$, $S$ will be the singular locus with dimension $n_{S}$. If $U$ is the smooth locus with inclusion map $j: U \hookrightarrow X$, let
$$IH^i_c(X, \Q) = H_c^{i-n}(X, j_{!*} \Q[n]).$$
Saito has constructed a mixed Hodge structure on this space. We will say more about this below.
The main result for this section is the following theorem.

\begin{thm}\label{HodgeThm}
    Suppose that $X$ is a normal quasi-projective variety, and let $f: \tX \rightarrow X$ be a resolution of singularities. 
  For all $i \in \Z$, there are natural isomorphisms
$$Gr^{0}_{F}IH^{i}_{c}(X, \C) \cong Gr^{0}_{F}H^{i}_{c}(\tX, \C) $$
$$Gr^{i}_{F}IH^{i}_{c}(X, \C) \cong Gr^{i}_{F}H^{i}_{c}(\tX, \C).$$
    If $R^{i}f_{*} \cO_{\tX} = 0$ for $1 \leq i \leq k$, then the natural maps
    $$Gr^{0}_{F}H^{i}_{c}(X, \C) \rightarrow Gr^{0}_{F}IH^{i}_{c}(X, \C)$$
    $$Gr^{i}_{F}H^{i}_{c}(X, \C) \rightarrow Gr^{i}_{F}IH^{i}_{c}(X, \C)$$
    are isomorphisms for $0 \leq i \leq k$ and injective for $i = k+1$.
\end{thm}

\begin{cor}
 If $X$ is normal and $R^{i}f_{*} \cO_{\tX} = 0$ for $1 \leq i \leq k$ the natural maps
    $$Gr^{0}_{F}H^{i}_{c}(X, \C) \rightarrow Gr^{0}_{F}H^{i}_{c}(\tX, \C)$$
$$Gr^{i}_{F}H^{i}_{c}(X, \C) \rightarrow Gr^{i}_{F}H^{i}_{c}(\tX, \C)$$
    are isomorphisms for $0 \leq i \leq k$ and injective for $i = k+1.$ 
\end{cor}

We will postpone this until after we introduce some basic constructions and lemmas. 
Given a smooth variety $X$, the Poincar\'e lemma gives an isomorphism
$$ \C_X\simeq  \Omega_X^\bullet. $$
The stupid filtration on the right induces a filtration on $H^*(X,\C)$, which we call the {\em de Rham filtration}. This need not coincide with the Hodge filtration unless $X$ is proper.
  Furthermore, the associated spectral sequence need not degenerate in general.
If $X$ is singular, we define the  de Rham filtration on $H^{i}(X, \C)$ and $H^{i}_{c}(X, \C)$  by
     $$^{d}F^{p}H^{i}(X, \C) = \im\bigg[H^{i-n}(X, F^{p}(\underline{\Omega}^{\bullet}_{X}[n] )) \rightarrow H^{i}(X, \C) \bigg]$$

     $$^{d}F^{p}H^{i}_{c}(X, \C) = \im\bigg[H^{i-n}_{c}(X, F^{p}(\underline{\Omega}^{\bullet}_{X}[n] )) \rightarrow H^{i}_{c}(X, \C) \bigg].$$

It will be convenient to extend this notion to mixed Hodge modules. To facilitate this,
choose an open embedding $\iota: X \hookrightarrow X'$ with $X'$ projective, and let $\rho: X' \hookrightarrow \cX$ be any closed embedding, where $\cX$ is a smooth projective variety of dimension $m$. We also denote $d = m - n.$ The composition of these maps will be denoted as $i_{X}: X \hookrightarrow \cX.$  There exists an open subset $Y \subset \cX$ such that $i:X \hookrightarrow Y$ is a closed embedding, and there is a commutative diagram
   $$\adjustbox{scale=1.25}{\bT X \arrow[r, "\iota"] \arrow[d, "i"] \arrow[dr, "i_{X}"] & X' \arrow[d, "\rho"] \\
         Y \arrow[r, "j"] & \cX. \eT}$$ 
    For any  object  $\cM \in D^{b}MHM(X)$ in the derived category of mixed Hodge modules, by \cite{saito} \cite{saito2} $H^{i}(X, DR(\cM))$ and $H^{i}_{c}(X, DR(\cM))$ have a mixed Hodge structure. By construction, there are filtered complexes
    $$(DR(i_{+}\cM), F_{\bullet}(DR(i_{+}\cM)))$$ 
    $$(DR(i_{X+}\cM), F_{\bullet}(DR(i_{X+}\cM)))$$
     $$(DR(i_{X!}\cM), F_{\bullet}(DR(i_{X!}\cM))).$$
     We can construct the Hodge and de Rham filtrations on $H^{i}(X, DR(\cM))$ and $H^{i}_{c}(X, DR(\cM))$ in the same way as above. The Hodge filtrations are as follows,
     $$F^{p}H^{i}(X, DR(\cM)) := \im\bigg[H^{i}(\cX,  F_{-p}(DR(i_{X+}\cM))) \rightarrow H^{i}(X, DR(\cM)) \bigg]$$

    $$F^{p}H^{i}_{c}(X, DR(\cM)) := \im\bigg[H^{i}(\cX,  F_{-p}DR(i_{X!}\cM)) \rightarrow H^{i}_{c}(X, DR(\cM)) \bigg].$$
    
    We define the de Rham filtration

    $$^{d}F^{p}H^{i}(X, DR(\cM)) := \im\bigg[H^{i}(Y, F_{-p}DR(i_{+}\cM)) \rightarrow H^{i}(X, DR(\cM)) \bigg]$$

     $$^{d}F^{p}H^{i}_{c}(X, DR(\cM)) := \im\bigg[H^{i}_{c}(Y, F_{-p}DR(i_{+}\cM)) \rightarrow H^{i}_{c}(X, DR(\cM)) \bigg].$$
     (We expect that this filtration is well defined, but we have not checked it since we will not need it.)
     
The Hodge filtration is strict, so the corresponding spectral sequence degenerates at $E_{1}$. However, the de Rham filtration may not be strict. So, the associated  sequences,
\begin{equation}\label{eq:deRham}
\begin{split}
   E^{p, i -p}_{1} = H^{i}(Y, Gr^{F}_{-p}DR(i_{+}\cM)) =H^{i}(X, Gr^{F}_{-p}DR(\cM))& \Rightarrow \hspace{.01in}^{d}Gr_{F}^{p}H^{i}(X, DR(\cM))\\
   E^{p, i -p}_{1} = H^{i}_{c}(Y, Gr^{F}_{-p}DR(i_{+}\cM)) =H^{i}_{c}(X, Gr^{F}_{-p}DR(\cM)) &\Rightarrow \hspace{.01in}^{d}Gr_{F}^{p}H^{i}_{c}(X, DR(\cM))
\end{split}
\end{equation}
    need not degenerate.

 \begin{rmk}
     If $\cM = \Q^{H}_{X}[n]$, then by  \cite{saito5}, the two filtered complexes
     $$F_{\bullet}DR(i_{+}\Q^{H}_{X}[n]) \quad \text{and} \quad i_{*}(F^{\bullet}(\underline{\Omega}_{X}^{\bullet}[n]))$$
     coincide. For each $p \in \Z,$
    $$Gr^{F}_{-p}DR(i_{+}\Q^{H}_{X}[n]) \simeq i_{*}(Gr^{p}_{F}\underline{\Omega}_{X}^{\bullet}[n]) = i_{*}(\underline{\Omega}^{p}_{X}[n-p]).$$
  
 \end{rmk}

  \begin{prop}
    For every $p \in \Z$, there exists  natural injective maps $$F^{p}H^{i}(X, DR(\cM)) \hookrightarrow \hspace{.01in} ^{d}F^{p}H^{i}(X, DR(\cM))$$
    $$^{d}F^{p}H^{i}_{c}(X, DR(\cM)) \hookrightarrow F^{p}H^{i}_{c}(X, DR(\cM)).$$
\end{prop}
\begin{proof}

 For $\cM \in D^{b}MHM(X)$, the inclusion map $j:Y \hookrightarrow \cX$ induces quasi-isomorphisms
 $$F_{-p}(DR(i_{X+}\cM))\vert_{Y} \simeq (F_{-p}DR(i_{+}\cM))$$
 $$F_{-p}(DR(i_{X!}\cM))\vert_{Y}  \simeq (F_{-p}DR(i_{+}\cM)).$$
 The adjunction maps induce natural morphisms
    $$F_{-p}(DR(i_{X+}\cM)) \rightarrow \bR j_{*} (F_{-p}DR(i_{+}\cM))$$
    $$\bR j_{!}(F_{-p}DR(i_{+}\cM)) \rightarrow F_{-p}(DR(i_{X!}\cM)).$$
    Thus, there are commutative diagrams
    $$\begin{tikzcd}
   H^{i}(\cX,  F_{-p}DR(i_{X+}\cM)) \arrow[r] \arrow[d] & H^{i}(X, DR(\cM)) \arrow[d, "id"] \\
   H^{i}(Y, F_{-p}DR(i_{+}\cM)) \arrow[r] & H^{i}(X, DR(\cM)) \end{tikzcd}$$
    $$\begin{tikzcd}
   H^{i}_{c}(Y, (F_{-p}DR(i_{+}\cM))) \arrow[r] \arrow[d] & H^{i}_{c}(X, DR(\cM)) \arrow[d, "id"]\\
   H^{i}(\cX,  F_{-p}DR(i_{X!}DR(\cM))) \arrow[r]  & H^{i}_{c}(X, DR(\cM)).
    \end{tikzcd}$$
    Therefore, 
   $$F^{p}H^{i}(X, DR(\cM)) = \im\bigg[H^{i}(\cX,  F_{-p}DR(i_{X+}DR(\cM))) \rightarrow H^{i}(X, DR(\cM)) \bigg]$$
   is a subset of 
   $$^{d}F^{p}H^{i}(X, DR(\cM)) = \im\bigg[H^{i}(Y,  (F_{-p}DR(i_{+}\cM))) \rightarrow H^{i}(X, DR(\cM)) \bigg].$$
   Similarly, for cohomology with compact support.
    \end{proof}

   \begin{defn}
       For an irreducible variety $X$ of dimension $n$,
       $$IH^{i}_{c}(X, \Q):= H^{i-n}_{c}(X, rat(IC^{H}_{X}))$$
       $$IH^{i}(X, \Q):= H^{i-n}(X, rat(IC^{H})).$$
       Also,
       $$ IH^{i}_{c}(X, \Q) \otimes \C \cong IH^{i}_{c}(X, \C) := H^{i-n}_{c}(X, DR(IC^{H}_{X}))$$
       $$IH^{i}(X, \Q) \otimes \C \cong IH^{i}(X, \C) := H^{i-n}(X, DR(IC^{H}_{X})),$$
       where $IC^{H}_{X}$ is the Hodge module corresponding to intersection cohomology.
   \end{defn}

   \begin{lemma}\label{WeightLemma}
        If $X$ is purely $n$-dimension, then 
       $$Gr^{W}_{j}H^{i}_{c}(X, \Q) = Gr^{W}_{j}IH^{i}_{c}(X, \Q) =  0 \quad \text{for $j>i$.}$$
   \end{lemma}

   \begin{proof}
      If $a_{X}:X \rightarrow \{pt\}$, then $\Q^{H}_{X}[n]= (a_{X}^{*}\Q^{H}_{pt})[n]$ and $Gr^{W}_{n}\cH^{0}(\Q^{H}_{X}[n]) = \bigoplus IC^{H}_{X_{\alpha}}$, where $X_{\alpha}$ are the irreducible components. By \cite{saito2},
       $$\cH^{i}(\Q^{H}_{X}[n]) = 0 \quad \text{for $i > 0$}$$
       $$Gr^{W}_{j}\cH^{i}(\Q^{H}_{X}[n]) = 0 \quad \text{for $j >i +n$.}$$
          Now, by \cite[Prop 2.26]{saito2}, if $\cM$ is a mixed Hodge module such that $Gr^{W}_{j}\cM = 0$ for $j > m$, then
       $$Gr^{W}_{j}H^{i-m}_{c}(X, rat(\cM)) = 0 \quad  \text{for $j > i.$}$$
       Since $Gr^{W}_{j}IC^{H}_{X_{\alpha}} = 0$ for $j > n$ and for each $\alpha$, we obtain
       $$Gr^{W}_{j}IH^{i}_{c}(X, \Q) = \bigoplus_{\alpha}Gr^{W}_{j}IH^{i}_{c}(X_{\alpha}, \Q)= 0 \quad \text{ for $j > i$.}$$
      For the last claim, because the weight $\Q^{H}_{X}[n]$ is less than equal to $n$, i.e.
      $$Gr^{W}_{j}\cH^{i}(\Q^{H}_{X}[n]) = 0 \quad \text{for $j >i +n$,}$$
      by \cite[(4.2.5)]{saito2},
      $$Gr^{W}_{j}\cH^{i-n}(a_{X!}\Q^{H}_{X}[n]) = Gr^{W}_{j}H^{i-n}_{c}(X, \Q[n]) = Gr^{W}_{j}H^{i}_{c}(X, \Q) = 0 \quad \text{for $j >i $.}$$
       \end{proof}

        \begin{lemma}\label{InjLemma}
            If $f: Z \rightarrow X$ is a proper surjective morphism, then the natural map
            $$Gr^{W}_{i}H^{i}_{c}(X,\Q) \rightarrow Gr^{W}_{i}H^{i}_{c}(Z,\Q)$$
            is injective.
        \end{lemma}

        \begin{proof}
            First, assume that $X$ and $Z$ are smooth and irreducible. If $\dim Z = m$, then by \cite{saito}, there is a non-canonical splitting
            $$f_{+}\Q^{H}_{Z}[m] \simeq \displaystyle \bigoplus_{i \in \Z} \cH^{i}(f_{+}\Q^{H}_{Z}[m])[-i].$$
            It is well known that $\Q^{H}_{X}[n]$ is a direct summand of $\cH^{-m +n}(f_{+}\Q^{H}_{Z}[m])$ \cite[Thm. 9.3.37]{max}. Hence $H^{i}_{c}(X, \Q)$ is a direct summand of $H^{i}_{c}(Z, \Q)$, and the map $H^{i}_{c}(X,\Q) \rightarrow H^{i}_{c}(Z,\Q)$ is injective. Since the functor $Gr^{W}_{i}(\bullet)$ is exact, we have a natural inclusion
            $$Gr^{W}_{i}H^{i}_{c}(X,\Q) \hookrightarrow Gr^{W}_{i}H^{i}_{c}(Z,\Q).$$

            Now, we proceed by induction on $\dim X$. The lemma holds whenever $\dim X = 0$. So, assume the lemma is true whenever $\dim X = k \geq 0.$ Let $\dim X = k+1$, and consider the following commutative diagram
            $$ \bT \tilde{Z} \arrow[d, "p"] \arrow[r] & \tX \arrow[d, "q"] & E \arrow[l] \arrow[d] \\
            Z \arrow[r, "f"] & X & S. \arrow[l] \eT$$
            Here $S$ is the singular locus, $q: \tX \rightarrow X$  is a resolution of singularities, $E = q^{-1}(S)_{red}$, and $p: \tilde{Z} \rightarrow Z$ is also a resolution of singularities. There is an exact sequence
            $$0 \rightarrow Gr^{W}_{i}H^{i}_{c}(X,\Q) \rightarrow Gr^{W}_{i}H^{i}_{c}(\tX,\Q) \oplus Gr^{W}_{i}H^{i}_{c}(S,\Q) \rightarrow Gr^{W}_{i}H^{i}_{c}(E,\Q).$$ 
            Since $\dim S \leq k$, by induction, the morphism
            $$Gr^{W}_{i}H^{i}_{c}(S,\Q) \rightarrow Gr^{W}_{i}H^{i}_{c}(E,\Q)$$
            is injective. Therefore, the morphism
             $$Gr^{W}_{i}H^{i}_{c}(X,\Q) \rightarrow Gr^{W}_{i}H^{i}_{c}(\tX,\Q)$$
             is also injective. Now, because of the following commutative diagram
             $$\bT Gr^{W}_{i}H^{i}_{c}(X,\Q) \arrow[r] \arrow[hook, d] & Gr^{W}_{i}H^{i}_{c}(Z,\Q) \arrow[hook, d] \\
             Gr^{W}_{i}H^{i}_{c}(\tX,\Q) \arrow[hook, r] & Gr^{W}_{i}H^{i}_{c}(\tilde{Z},\Q), \eT$$
             we have the lemma to hold by induction. 
            
        \end{proof}
        
        It is important to note that if $X$ is an irreducible variety and $f: \tX \rightarrow X$ is a resolution of singularities, then there exists a factorization
        \begin{equation}\label{eq:weber}
       \bT \Q_{X}[n] \ar[r] \ar[dr, "f^{*}"] & IC_{X} \ar[d]\\
        & \bR f_{*}\Q_{\tX}[n].\eT
        \end{equation}
 See \cite{Weber} or \cite{BBJFKGK} for details. A more general result was shown. Consider a morphism $f:Z \rightarrow X$ between irreducible algebraic varieties. Say $\dim(Z) = m$ and $\dim(X) = n$ with $d = m-n$. It was shown in \cite{BBJFKGK} \cite{Weber} that there is morphism $\lambda: IC_{X} \rightarrow \bR f_{*}IC_{Z}[-d]$ such that the following diagram commutes
$$\bT IC_{X} \arrow[r, "\lambda"] & \bR f_{*}IC_{Z}[-d] \\
\Q_{X}[n] \arrow[u, "\alpha_{X}"] \arrow[r, "f^*"] & (\bR f_{*} \Q_{Z}[m])[-d] \arrow[u, "\bR f_{*}\alpha_{Z}"] \eT$$ 
where $f^{*}: \Q_{X}[n] \rightarrow (\bR f_{*}\Q_{Z}[m])[-d]$ is the natural pull-back map, $\alpha_{X}: \Q_{X}[n] \rightarrow IC_{X}$ is the natural quotient map, and $\bR f_{*}\alpha_{Z}: (\bR f_{*}\Q_{Z}[m])[-d] \rightarrow \bR f_{*}IC_{Z}[-d]$ is the push-forward of the quotient map from $Z$. We will follow the proof given by Weber \cite{Weber} to show the commutative diagram above extends to the derived category of mixed Hodge modules.

\begin{thm}\label{Weber'sExt}
    Let $f:Z \rightarrow X$ be a morphism between irreducible algebraic varieties. Say $\dim(Z) = m$ and $\dim(X) = n$ with $d = m-n$.  There exists a morphism $\lambda_{+}: IC^{H}_{X} \rightarrow  f_{+}IC^{H}_{Z}[-d]$ such that the following diagram commutes
$$\bT IC^{H}_{X} \arrow[r, "\lambda_{+}"] &  f_{+}IC^{H}_{Z}[-d] \\
\Q^{H}_{X}[n] \arrow[u, "\alpha_{X}"] \arrow[r, "f^*"] & ( f_{+} \Q^{H}_{Z}[m])[-d] \arrow[u, "f_{+}\alpha_{Z}"] \eT$$
\end{thm}

\begin{proof}
    We follow the same argument given in \cite{Weber}. Let $\pi_{X}: \tX \rightarrow X$ be any resolution of singularities of $X$. We will denote the reduced fiber product as $\tilde{Z} = (Z \times_{X} \tX)_{red}$ with the proper map $\pi_{Z}: \tilde{Z} \rightarrow Z$. Note that $\tilde{Z}$ may be singular and not equidimensional. From the commutative diagram
    $$\bT \tilde{Z} \arrow[r, "\tilde{f}"] \arrow[d, "\pi_{Z}"] & \tX \arrow[d, "\pi_{X}"]\\
    Z \arrow[r, "f"] & X \eT$$
    we obtain the following diagram in the derived category of mixed Hodge modules
    $$\bT \pi_{X+}IC^{H}_{\tX}  \arrow[r, "="] & \pi_{X+}\Q^{H}_{\tX}[n] \arrow[r, "\pi_{X+}(\tilde{f}^{*})"] & f_{+}\pi_{Z+}(\Q^{H}_{\tilde{Z}}[m])[-d] \arrow[rr, "f_{+}\pi_{Z+}(\alpha_{\tilde{Z}})"] & &f_{+}\pi_{Z+}IC^{H}_{\tilde{Z}}[-d] \arrow[d, dotted, "?"]\\
    IC^{H}_{X}   \arrow[u, dotted, "?"] & \arrow[l, "\alpha_{X}"]\Q^{H}_{X}[n] \arrow[r, "f^{*}"] \arrow[u, "\pi^{*}_{X}"] & f_{+}(\Q^{H}_{Z}[m])[-d] \arrow[rr, "f_{+}(\alpha_{Z})"]  \arrow[u, "f_{+}(\pi_{Z}^{*})"]& &f_{+}IC^{H}_{Z}[-d] \eT$$
    To prove the existence of $\lambda_{+}: IC^{H}_{X} \rightarrow f_{+}IC^{H}_{Z}[-d],$ we need to show the arrows with question marks exist in such a way that the diagram above commutes. It is well known that $IC^{H}_{X}$ is a direct summand of $\pi_{X+}IC^{H}_{\tX}$ (up to isomorphism), see \cite{max} for details. Similarly, $IC^{H}_{Z}$ is a direct summand of $\pi_{Z+}IC^{H}_{\tilde{Z}}$. If we choose the inclusion map 
    $$i: IC^{H}_{X} \hookrightarrow \pi_{X+}IC^{H}_{\tX}$$
    and the projection map (up to a shift) 
    $$f_{+}(p):f_{+}\pi_{Z+}IC^{H}_{\tilde{Z}}[-d] \rightarrow f_{+}IC^{H}_{Z}[-d]$$
    the diagram above commutes. Indeed, consider the two maps
    $$i \circ\alpha_{X}: \Q^{H}_{X}[n] \rightarrow \pi_{X+}\Q^{H}_{\tX}[n]$$
    $$\pi^{*}_{X}:\Q^{H}_{X}[n] \rightarrow \pi_{X+}\Q^{H}_{\tX}[n]$$
    If we denote $U \subset X$ to be the open subset such that $\pi^{-1}_{X}(U) \rightarrow U$ is an isomorphism, then 
    $$(i \circ \alpha_{X})|_{U} = \pi^{*}_{X}|_{U}.$$
    If $a_{X}:X \rightarrow \{pt\}$ is the natural map, then
    $$Hom(\Q^{H}_{X}[n], \pi_{X+}\Q^{H}_{\tX}[n]) = Hom(a^{*}_{X}\Q^{H}_{pt}[n],\pi_{X+}\Q^{H}_{\tX}[n]) = Hom(\Q^{H}_{pt}[n],a_{X+}\pi_{X+}\Q^{H}_{\tX}[n]) = H^{0}(\tX, \Q)$$
    If $\pi_{X}^{-1}(U) = V$, then
    $$Hom(\Q^{H}_{U}[n], \pi_{X+}\Q^{H}_{V}[n]) = Hom(a^{*}_{U}\Q^{H}_{pt}[n],\pi_{X+}\Q^{H}_{V}[n]) = Hom(\Q^{H}_{pt}[n],a_{U+}\pi_{X+}\Q^{H}_{V}[n]) = H^{0}(V, \Q)$$
    Since there is an inclusion $H^{0}(\tX,\Q) \subseteq H^{0}(V, \Q)$, and the two maps $i \circ \alpha_{X}$ and $\pi^{*}_{X}$ agree on $U$, we have $i \circ \alpha_{X} = \pi^{*}_{X}$. A similar argument can be made for 
    $$f_{+}(p):f_{+}\pi_{Z+}IC^{H}_{\tilde{Z}}[-d] \rightarrow f_{+}IC^{H}_{Z}[-d].$$
\end{proof}

\begin{cor}\label{factorCor}
    If $X$ is an irreducible variety and $f: \tX \rightarrow X$ is a resolution of singularities, then there exists the following factorization
        $$
       \bT \Q^{H}_{X}[n] \ar[r] \ar[dr, "f^{*}"] & IC^{H}_{X} \ar[d]\\
        & f_{+}\Q^{H}_{\tX}[n]\eT$$
        that is compatible with \ref{eq:weber}.
\end{cor}

        \begin{prop}\label{injProp}
            If $X$ is an irreducible variety, then the natural map
            $$Gr^{W}_{i}H^{i}_{c}(X, \Q) \rightarrow Gr^{W}_{i}IH^{i}_{c}(X, \Q)$$
            is injective.
        \end{prop}

        \begin{proof}
            If $f: \tX \rightarrow X$ is a resolution of singularities, then the map
            $$Gr^{W}_{i}H^{i}_{c}(X, \Q) \rightarrow Gr^{W}_{i}H^{i}_{c}(\tX, \Q)$$
            is injective by Lemma \ref{InjLemma}. There is also a factorization
            $$\bT Gr^{W}_{i}H^{i}_{c}(X, \Q) \ar[r] \ar[dr, hook] & Gr^{W}_{i}IH^{i}_{c}(X, \Q) \ar[d] \\
            & Gr^{W}_{i}H^{i}_{c}(\tX, \Q) \eT$$
             by Corollary \ref{factorCor}. Therefore, the map $Gr^{W}_{i}H^{i}_{c}(X, \Q) \rightarrow Gr^{W}_{i}IH^{i}_{c}(X, \Q)$ must be injective.
        \end{proof}

        \begin{rmk}
            Recently, it was shown by Park and Popa \cite{ParkPopa} that if $K:= Cone(\Q^{H}_{X}[n] \rightarrow IC^{H}_{X})$, then $K$ is of weight $\leq n$. Equivalently,
            $$Gr^{W}_{i}H^{j}(K) = 0 \quad \text{for $i> j +n.$}$$
            So, consider a resolution of singularities $f: \tX \rightarrow X$ and the exact traingle
            $$\bT \bR Hom(K, f_{+}\Q^{H}_{\tX}[n]) \arrow[r] & \bR Hom(IC^{H}_{X}, f_{+}\Q^{H}_{\tX}[n]) \arrow[r] &  \bR Hom(\Q^{H}_{X}[n], f_{+}\Q^{H}_{\tX}[n]) \arrow[r, "+1"] & \hfill \eT$$
            Since $K$ is of weight $\leq n$ and $f_{+}\Q^{H}_{\tX}[n]$ is of pure weight $n$, we have $Ext^{1}(K, f_{+}\Q^{H}_{\tX}[n]) = 0$ \cite{saito2}. Therefore, we have a surjective map
            $$\bT Hom(IC^{H}_{X}, f_{+}\Q^{H}_{\tX}[n]) \arrow[r] &   Hom(\Q^{H}_{X}[n], f_{+}\Q^{H}_{\tX}[n]) \arrow[r] & 0. \eT$$ 
            So, there exists a map $IC^{H}_{X} \dashrightarrow f_{+}\Q^{H}_{\tX}[n]$ such that the diagram below commutes
            $$
       \bT \Q_{X}^{H}[n] \ar[r, "\alpha_{X}"] \ar[dr, "f^{*}"] & IC^{H}_{X} \ar[d, dashed]\\
        & f_{+}\Q_{\tX}[n].\eT$$
        By Theorem \ref{Weber'sExt}, we know we can take $IC^{H}_{X} \dashrightarrow f_{+}\Q^{H}_{\tX}[n]$ to be the inclusion map (up to isomorphism).
        \end{rmk}

\begin{lemma}\label{lemma:F0DR}
 We have isomorphisms
 $$Gr_0^F DR(\Q_X^{H}[n]) \cong \underline{\Omega}_X^{0}[n]$$
 $$Gr_0^F DR(IC_X^H) \cong {\bf R} f_*\cO_{\tilde X}[n].$$
 
\end{lemma}

\begin{proof}
 These isomorphisms are due to Saito. See \cite{saito5} for the first isomorphism and \cite{saito3} for the second.
\end{proof}

\begin{proof}[Proof of Theorem \ref{HodgeThm} ]
   If $K:= Cone \bigg(\Q^{H}_{X}[n] \rightarrow IC^{H}_{X} \bigg)$, then we have the exact triangle 
   $$ \bT \Q^{H}_{X}[n] \arrow[r] & IC^{H}_{X} \arrow[r] & K \arrow[r, "+1"] & \hfill. \eT$$
   This exact triangle induces the long exact sequence
   $$\cdots \rightarrow H^{i}_{c}(X, \C) \rightarrow IH^{i}_{c}(X, \C) \rightarrow H^{i-n}_{c}(X, DR(K)) \rightarrow \cdots. $$
   The morphisms are morphisms of mixed Hodge structures. Hence, there is a long exact sequence
   $$\cdots \rightarrow Gr
   ^{0}_{F}H^{i}_{c}(X, \C) \rightarrow Gr
   ^{0}_{F}IH^{i}_{c}(X, \C) \rightarrow Gr
   ^{0}_{F}H^{i-n}_{c}(X, DR(K)) \rightarrow \cdots. $$
   We will first show $Gr
   ^{0}_{F}H^{i-n}_{c}(X, DR(K)) = 0$ for $0 \leq i \leq k.$ Note that we have the quasi-isomorphisms 
   $$F_{0}DR(\Q^{H}_{X}[n]) \simeq DR(\Q^{H}_{X}[n])$$ 
   $$F_{0}DR(IC^{H}_{X}) \simeq DR(IC^{H}_{X}).$$
   Therefore, we must also have $F_{0}DR(K) \simeq DR(K)$ and $^{d}F^{0}H^{i-n}_{c}(X, DR(K)) = H^{i-n}_{c}(X, DR(K))$. Hence, there is a commutative diagram
   $$\begin{tikzcd}
        ^{d}F^{0}H^{i-n}_{c}(X, DR(K)) = H^{i-n}_{c}(X, DR(K)) \arrow[r] \arrow[d] & ^{d}Gr^{0}_{F}H^{i-n}_{c}(X, DR(K)) \arrow[d] \arrow[r] &  0\\
         F^{0}H^{i-n}_{c}(X, DR(K)) =  H^{i-n}_{c}(X, DR(K)) \arrow[r] & Gr^{0}_{F}H^{i-n}_{c}(X, DR(K)) \arrow[r] & 0.
    \end{tikzcd}$$
    By Lemma \ref{lemma:F0DR}, we have a distinguished triangle
    $$\underline{\Omega}_X^{0}[n]\to {\bf R} f_*\cO_{\tilde X}[n]\to Gr^{F}_{0}DR(K)\xrightarrow{+1}. $$
    If $X$ is normal and $R^{i}f_{*}\cO_{\tX} = 0$ for $ 1 \leq i \leq k$, then by Theorem \ref{VanishingThm}, 
    $$\cH^{i}(Gr^{F}_{0}DR(K)) = 0 \quad \text{for $-n \leq i \leq -n+k$.}$$
    Thus, for $0 \leq i \leq k$, $H^{i-n}(X, (Gr^{F}_{0}DR(K)) = 0$ by applying the spectral sequence
    $$E^{p,q}_{2} = H^{p}(X, \cH^{q}(Gr^{F}_{0}DR(K))) \Rightarrow H^{p+q}(X, Gr^{F}_{0}DR(K)).$$
    By the spectral sequence \eqref{eq:deRham}, we obtain  $^{d}Gr^{0}_{F}H^{i-n}_{c}(X, DR(K)) = 0$ for $0 \leq i \leq k.$ By the commutative diagram above, the natural map
    $$ ^{d}Gr^{0}_{F}H^{i-n}_{c}(X, DR(K)) \rightarrow Gr^{0}_{F}H^{i-n}_{c}(X, DR(K)) $$
    is surjective. So we must also have $Gr^{0}_{F}H^{i-n}_{c}(X, DR(K)) = 0$ for $0 \leq i \leq k.$

   For $0 \leq i \leq k$ we have
    $$\dim Gr^{i}_{F}Gr^{W}_{i}H^{i-n}_{c}(X, DR(K)) = \dim Gr^{0}_{F}Gr^{W}_{i}H^{i-n}_{c}(X, DR(K)) = \dim Gr^{W}_{i}Gr^{0}_{F}H^{i-n}_{c}(X, DR(K))) = 0,$$
    where the first equality follows from the purity of $Gr^{W}_{i}H^{i-n}_{c}(X, DR(K))$.
    So, the natural map
    $$Gr^{i}_{F}Gr^{W}_{i}H^{i}_{c}(X, \C) \rightarrow Gr^{i}_{F}Gr^{W}_{i}IH^{i}_{c}(X, \C)$$
    is an isomorphism for $0 \leq i \leq k$. For cohomology with compact support, there are exact sequences 
    $$0 \rightarrow W_{i-1}H^{i}_{c}(X, \C) \rightarrow H^{i}_{c}(X, \C) \rightarrow Gr^{W}_{i}H^{i}_{c}(X, \C) \rightarrow 0$$
    $$0 \rightarrow W_{i-1}IH^{i}_{c}(X, \C) \rightarrow IH^{i}_{c}(X, \C) \rightarrow Gr^{W}_{i}IH^{i}_{c}(X, \C) \rightarrow 0.$$

    Therefore, we obtain the exact sequences
    $$0 \rightarrow Gr^{i}_{F}W_{i-1}H^{i}_{c}(X, \C) \rightarrow Gr^{i}_{F}H^{i}_{c}(X, \C) \rightarrow Gr^{i}_{F}Gr^{W}_{i}H^{i}_{c}(X, \C) \rightarrow 0$$
    $$0 \rightarrow Gr^{i}_{F}W_{i-1}IH^{i}_{c}(X, \C) \rightarrow Gr^{i}_{F}IH^{i}_{c}(X, \C) \rightarrow Gr^{i}_{F}Gr^{W}_{i}IH^{i}_{c}(X, \C) \rightarrow 0.$$
    Since $F^{0}H^{i}_{c}(X, \C) = H^{i}_{c}(X, \C)$ and $F^{0}IH^{i}_{c}(X, \C) = IH^{i}_{c}(X, \C)$, the theory of weights for mixed Hodge structures forces $Gr^{i}_{F}W_{i-1}H^{i}_{c}(X, \C) = 0$ and $Gr^{i}_{F}W_{i-1}IH^{i}_{c}(X, \C)=0.$  The second pair of isomorphisms of the  theorem now follows from the previous proposition and the isomorphisms
    $$Gr^{i}_{F}H^{i}_{c}(X, \C) \cong Gr^{i}_{F}Gr^{W}_{i}H^{i}_{c}(X, \C) \cong  Gr^{i}_{F}Gr^{W}_{i}IH^{i}_{c}(X, \C) \cong Gr^{i}_{F}IH^{i}_{c}(X, \C).$$
    
    To complete the proof, note that by Lemma \ref{lemma:F0DR},
    $$Gr^{F}_{0}DR(IC^{H}_{X}) \simeq Gr^{F}_{0}DR(f_{+}\Q^{H}_{\tX}) \simeq \bR f_{*}\cO_{\tX}[n].$$
    So, if $\tilde{K}:= Cone \bigg(IC^{H}_{X} \rightarrow f_{+}\Q^{H}_{\tX}[n]  \bigg)$, we must have $Gr^{F}_{0}DR(\tilde{K}) \simeq 0.$ Thus, by using the same argument as above, we obtain
       $$Gr^{0}_{F}IH^{i}_{c}(X, \C) \cong  Gr^{0}_{F}H^{i}_{c}(\tX, \C) \quad \text{and} \quad 
    Gr^{i}_{F}IH^{i}_{c}(X, \C) \cong Gr^{i}_{F}H^{i}_{c}(\tX, \C).$$
\end{proof}



\begin{cor}
     Suppose that $X$ is a normal variety, and let $f: \tX \rightarrow X$ be a resolution of singularities. If $\tilde{S}$ is the the smallest closed subset of $X$ such that $X \backslash \tilde{S}$ has rational singularities, then the natural maps
    $$Gr^{0}_{F}H^{i}_{c}(X, \C) \rightarrow Gr^{0}_{F}IH^{i}_{c}(X, \C)$$
    $$Gr^{i}_{F}H^{i}_{c}(X, \C) \rightarrow Gr^{i}_{F}IH^{i}_{c}(X, \C)$$
    are isomorphisms for $0 \leq i \leq  \depth_{\tilde{S}}(\cO_{X}) -2$ and injective for $i =\depth_{\tilde{S}}(\cO_{X}) -1.$
\end{cor}

\begin{proof}
    The result follows from Theorem \ref{HodgeThm} and Corollary \ref{rat.rmk}.
    \end{proof}

\begin{cor}
    Suppose that $X$ is a normal projective variety, and let $f: \tX \rightarrow X$ be a resolution of singularities. If $R^{i}f_{*} \cO_{\tX} = 0$ for $1 \leq i \leq k$, then the natural maps
    $$H^{i}(X, \underline{\Omega}^{0}_{X}) \rightarrow  H^{i}(\tX, \cO_{\tX})$$
    $$H^{0}(X, \underline{\Omega}^{i}_{X}) \rightarrow H^{0}(\tX, \Omega_{\tX}^{i})$$
    are isomorphisms for $0 \leq i \leq k$ and injective $i = k+1.$ 
\end{cor}

\begin{rmk}
Consider when $X$ is a normal projective variety with $R^{i}f_{*} \cO_{\tX} = 0$ for $1 \leq i \leq k$, where $f: \tX \rightarrow X$ is a resolution of singularities. By the previous corollary, we have the following identifications of Hodge numbers for $0 \leq i \leq k$.
    $$h^{0,i}(X)= \dim H^{i}(X, \underline{\Omega}^{0}_{X}) =\dim  H^{i}(\tX, \cO_{\tX}) = h^{0,i}(\tX)$$
    $$h^{i,0}(X) = \dim H^{0}(X, \underline{\Omega}^{i}_{X}) = H^{0}(\tX, \Omega_{\tX}^{i}) = h^{i,0}(\tX).$$
    Since $\tX$ is smooth, we must also have
    $$h^{0,i}(X) =h^{0,i}(\tX) = h^{i,0}(\tX)= h^{i,0}(X) \quad \text{for $0 \leq i \leq k.$}$$
    For the case when $i = k+1,$ we have inequalities
    $$h^{0,k+1}(X) \leq  h^{0,k+1}(\tX) \quad \text{and} \quad h^{k+1,0}(X)\leq  h^{k+1,0}(\tX).$$
    \end{rmk}
\begin{thm}
Suppose $X$ is a normal quasi-projective variety and $f: \tX \rightarrow X$ is a resolution of singularities. If $R^{1}f_{*}\cO_{\tX} = 0$, then the natural map $H^{i}_{c}(X,\Q) \rightarrow IH^{i}_{c}(X,\Q)$ is an isomorphism for $i \leq 1$ and injective for $i = 2.$ In particular, if $X$ is projective, then $H^{i}(X, \Q)$ has a pure Hodge structure for $i \leq 2$ .
\end{thm}

\begin{proof}
     Recall from the proof of Theorem \ref{HodgeThm} we have the long exact sequence
     $$\cdots \rightarrow H^{i}_{c}(X, \C) \rightarrow IH^{i}_{c}(X, \C) \rightarrow H^{i-n}_{c}(X, DR(K)) \rightarrow H^{i+1}_{c}(X, \C) \rightarrow \cdots. $$
     We know by Proposition \ref{injProp} $Gr^{W}_{j}H^{i-n}_{c}(X, DR(K)) = 0$ for $j>i.$ Thus
     $$Gr^{W}_{j}H^{1-n}_{c}(X, DR(K)) = \begin{cases}
        \displaystyle  \bigoplus_{0 \leq p \leq j}Gr^{p}_{F}Gr^{W}_{j}H^{1-n}_{c}(X, DR(K)) & \text{for $j \leq 1$} \\ \\
         0 & \text{otherwise}
     \end{cases} $$
     $$Gr^{W}_{j}H^{0-n}_{c}(X, DR(K)) = \displaystyle \begin{cases}
         \displaystyle Gr^{0}_{F}Gr^{W}_{0}H^{0-n}_{c}(X, DR(K)) \\ \\
         0 & \text{otherwise}
     \end{cases} $$
     By the proof of Theorem \ref{HodgeThm}, we have $Gr^{W}_{j}H^{1-n}_{c}(X, DR(K)) = Gr^{W}_{j}H^{0-n}_{c}(X, DR(K))$ for all $j \in \Z.$ So, we must have $H^{1-n}_{c}(X, DR(K))= H^{0-n}_{c}(X, DR(K)) = 0$, which implies $H^{i}_{c}(X,\Q) \rightarrow IH^{i}_{c}(X, \Q)$ is an isomorphism for $i \leq 1$ and injective for $i = 2$. If $X$ is projective, then $IH^{i}(X, \Q)$ has a pure Hodge structure, which implies $H^{i}(X, \Q)$ also has a pure Hodge structure for $i \leq 2$.
     
\end{proof}


        


     \begin{prop}\label{ExtProp}
   For $i +p \leq n-n_{S} -2,$
   $$\cH^{i+p -n}(Gr^{F}_{-p}DR(IC^{H}_{X})) \cong R^{i}j_{*}\Omega^{p}_{U}.$$
\end{prop}

\begin{proof}
    By \cite{hiatt1}, there is an isomorphism
    $$ \cH^{i+p -n}(Gr^{F}_{-p}DR(IC^{H}_{X})) \cong R^{i}f_{*}\Omega^{p}_{\tX}(\log E) \quad \text{for $i +p\leq n-n_{S} -1.$}$$
    But, by Corollary \ref{LogCor}, we also have $R^{i}f_{*}\Omega^{p}_{\tX}(\log E) \cong R^{i}j_{*}\Omega^{p}_{U} $ for $ i +p \leq n-n_{S}  -2.$
\end{proof}

\begin{thm}\label{LCIThm}
    Let $X$ be an irreducible local complete intersection. 
    \begin{enumerate}

    \item The natural morphism of mixed Hodge structures
    $$H^{i}_{c}(X, \Q) \rightarrow IH^{i}_{c}(X, \Q) $$
    $$H^{i}(X, \Q) \rightarrow IH^{i}(X, \Q)$$
    are isomorphisms  for $0 \leq i \leq n - n_{S} -2$ and injective for $i = n-n_{S} -1;$\\

    \item  If $X$ is projective, then $H^{i}(X,\Q)$ has a pure Hodge structure of weight $i$ whenever $0 \leq i \leq n - n_{S} -1$. Furthermore, whenever $i \leq n - n_{S} -2$,
        $$H^{i}(X, \C) \cong \bigoplus_{p + q = i}H^{q}(X, \underline{\Omega}^{p}_{X}) \cong \bigoplus_{p + q = i}H^{q}(X, \cH^{0}(\underline{\Omega}^{p}_{X})) \cong \bigoplus_{p + q = i}H^{q}(U, \Omega^{p}_{U}) .$$
    \end{enumerate}
   
\end{thm}

\begin{proof}
     Since $X$ is a local complete intersection, $\Q^{H}_{X}[n]$ is a mixed Hodge module, and we have the short exact sequence
     $$0 \rightarrow W_{n-1}\Q^{H}_{X}[n] \rightarrow \Q^{H}_{X}[n] \rightarrow IC^{H}_{X} \rightarrow 0.$$
     By \cite[Lemma 5.10]{hiatt1},
     $$\cH^{i}(Gr^{F}_{-p}DR(W_{n-1}\Q^{H}_{X}[n])) = 0 \quad \text{for $p \in \Z$ and $i  < - n_{S}.$}$$
     By the spectral sequence \eqref{eq:deRham}, we obtain
    $$^{d}Gr^{p}_{F}H^{i-n}_{c}(X, DR(W_{n-1}\Q^{H}_{X}[n])) = 0 \quad \text{for $p \in \Z$ and $i \leq n - n_{S} -1.$}$$
        $$^{d}Gr^{p}_{F}H^{i-n}(X, DR(W_{n-1}\Q^{H}_{X}[n])) = 0 \quad \text{for $p \in \Z$ and $i \leq n - n_{S} -1.$}$$
        Thus $H^{i-n}_{c}(X, DR(W_{n-1}\Q^{H}_{X}[n])) = H^{i-n}(X, DR(W_{n-1}\Q^{H}_{X}[n]))= 0$ for $i \leq n - n_{S} -1.$

     If $X$ is projective, $ IH^{i}(X, \C)$ has a pure Hodge structure. By $(1)$, whenever $i \leq n - n_{S} -1$ we always have an inclusion map $H^{i}(X, \C) \hookrightarrow IH^{i}(X, \C)$. Thus $H^{i}(X, \C)$ has a pure Hodge structure for $i \leq n - n_{S} -1.$ Also, by \cite{MuPo}, $\cH^{i}(\underline{\Omega}^{p}_{X}) \cong 0$ for $1 \leq i \leq n -n_{S} - p -2.$ Therefore, by Proposition \ref{ExtProp},
      $$0 = \cH^{i}(\underline{\Omega}^{p}_{X}) \cong \cH^{i+p -n}(Gr^{F}_{-p}DR(IC^{H}_{X})) \cong  R^{i}j_{*}\Omega^{p}_{U} \quad \text{for $1 \leq i \leq n - n_{S} - p -2.$}$$
      $$\cH^{0}(\underline{\Omega}^{p}_{X}) \cong j_{*}\Omega^{p}_{X} \quad \text{for $p\leq n - n_{S} -2.$}$$
       So, for $0 \leq i \leq n- n_{S} -2,$
        $$H^{i}(X, \C) \cong \bigoplus_{p + q = i}H^{q}(X, \underline{\Omega}^{p}_{X}) \cong \bigoplus_{p + q = i}H^{q}(X, \cH^{0}(\underline{\Omega}^{p}_{X})) \cong \bigoplus_{p + q = i}H^{q}(X, \bR j_{*}\Omega^{p}_{U}) \cong \bigoplus_{p + q = i}H^{q}(U, \Omega^{p}_{U}).$$
\end{proof}

\begin{rmk}
    In the case of a local complete intersection, Theorem \ref{LCIThm} agrees with a result by Park and Popa \cite[Cor. 11.9]{ParkPopa}.
\end{rmk}

\section{On a conjecture of Musta\c{t}\u{a} and Popa}

Recall that we have a closed embedding $i: X \hookrightarrow Y$ into a smooth variety of dimension $m$. If we consider $\D(i_{+}\Q^{H}_{X}[n]) \in D^{b}MHM_{X}(Y),$ the dual of $i_{+}\Q^{H}_{X}[n]$, then the underlying left $\cD_{Y}$-module of the mixed Hodge module $\cH^{i}(\D(i_{+}\Q^{H}_{X}[n]))$ is precisely the local cohomology sheaf $\cH^{i+d}_{X}(\cO_{Y})$, where $d = \dim Y - \dim X.$ A fascinating result of Musta\c{t}\u{a} and Popa describes the \emph{local cohomological dimension $X$}
 $$\lcd(Y,X) = max \{q \hspace{.05in} \vert \hspace{.05in}  \cH^{q}_{X}(\cO_{Y}) \neq 0 \}$$
 in terms of log resolutions and log forms.

 \begin{thm}\cite{MuPo}\label{MuPoThm}
    Let $X$ be a closed subscheme of $Y$, and $c$ a positive integer. Then the following are equivalent:
    \begin{enumerate}
        \item $\lcd(Y,X) \leq c$

        \item For any (some) log resolution $f: \tilde{Y} \rightarrow Y$ of the pair $(Y,X)$, assumed to be an isomorphism over the complement of $X$ in $Y$, if $D = f^{-1}(X)_{red}$, then 
        $$R^{i+j}f_{*}\Omega^{m-i}_{\tilde{Y}}(\log D) = 0 \quad \text{for all $j \geq c, i \geq 0$}.$$
    \end{enumerate}
 \end{thm}
 Musta\c{t}\u{a} and Popa propose the following conjecture, which concerns the depth of $\cO_{X}$.
 \begin{conj}\label{DepthConj.}
     If $\depth(\cO_{X}) \geq i +2$, then $R^{m-2}f_{*}\Omega^{m-i}_{\tilde{Y}}(\log D) = 0.$
 \end{conj}
 Here, we have $\depth(\cO_{X})$ to be taken over all closed points of $X$. The conjecture was proven by Musta\c{t}\u{a} and Popa in the case of isolated singularities \cite{MuPo}. Although it is known that the ideal pattern
$$\depth(\cO_{X}) \geq k \Rightarrow \lcd(Y, X) \leq m -k$$
stops when $k =3$, observe that  the conjecture does relate $\depth(\cO_{X})$ with $\lcd(Y, X)$ because of Theorem \ref{MuPoThm}. 

 Conjecture \ref{DepthConj.} may be stated in the terms of the complex $\underline{\Omega}^{i}_{X}.$ It was shown in \cite{MuPo} that there is an isomorphism
$$R^{m-2}f_{*}\Omega^{m-i}_{\tilde{Y}}(\log D) \cong \cH^{-1}(\bR \cH om_{Y}(\underline{\Omega}^{i}_{X}, \omega_{Y})).$$ If $x \in X$ is a closed point, recall that we define 
$$\depth((\underline{\Omega}^{i}_{X})_{x}) := \min\{i \hspace{.03in} | \hspace{.03in} \cH^{-i}(\bR \cH om_{X}((\underline{\Omega}^{i}_{X}, \omega_{X}^{\bullet}))_{x})\neq 0\} = \min \{i \hspace{.03in} | \hspace{.03in} \cH^{i}(\bR \underline{\Gamma_{x}}((\underline{\Omega}^{i}_{X})_{x}) \neq 0 \}$$ 
$$\depth(\underline{\Omega}^{i}_{X}):= \min_{x \in X}\depth((\underline{\Omega}^{i}_{X})_{x}).$$

\begin{lemma}\cite{MuPo}
    If $X$ is reduced and $\dim X \geq 2$, then for any $k \geq 0$ we have an equivalence
    $$\text{$R^{m-2}f_{*}\Omega^{m-i}_{\tilde{Y}}(\log D) = 0$ if and only if $\depth(\underline{\Omega}^{i}_{X}) \geq 2.$}$$
\end{lemma}

\begin{thm}\label{MuPoConj}
    Suppose that $X$ is a normal variety, and let $f: \tX \rightarrow X$ be a resolution of singularities. If $R^{i}f_{*} \cO_{\tX} = 0$ for $1 \leq i \leq k$, then $$\depth(\underline{\Omega}^{i}_{X}) \geq 2 \quad \text{for $i \leq k.$}$$
\end{thm}

\begin{proof}
    Since the problem is local, we may assume $X$ is affine. Also, it suffices to consider when $k \leq n -2$ because of a result of Kebekus and Schnell \cite[Thm. 1.10]{ks}.

   If $K:= Cone \bigg(\Q^{H}_{X}[n] \rightarrow IC^{H}_{X} \bigg)$, then we have the exact triangle 
   $$ \bT \Q^{H}_{X}[n] \arrow[r] & IC^{H}_{X} \arrow[r] & K \arrow[r, "+1"] & \hfill. \eT$$
   By applying $Gr^{F}_{-i}DR(\bullet),$ we have the following exact triangle   
$$
\bT \underline{\Omega}^{i}_{X}[n-i] \arrow[r]  & Gr^{F}
_{-i}DR(IC^{H}_{X}) \arrow[r] & Gr^{F}_{-i}DR(K) \arrow[r, "+1"] & \hfill. \eT$$
To help with notation, we have the following definitions
$$I\underline{\Omega}^{i}_{X}:=Gr^{F}
_{-i}DR(IC^{H}_{X})[i-n]$$
$$\Omega^{i}(K):= Gr^{F}_{-i}DR(K)[i-n].$$
Thus, we have the exact triangle
\begin{equation}\label{eq:DRtriangle}
\bT \underline{\Omega}^{i}_{X} \arrow[r]  & I\underline{\Omega}^{i}_{X} \arrow[r] & \Omega^{i}(K) \arrow[r, "+1"] & \hfill. \eT
\end{equation}
Note that $\cH^{0}(I\underline{\Omega}^{i}_{X}) \cong f_{*}\Omega^{i}_{\tX}$, see \cite[\S 8]{ks} for a proof. Both $\cH^{0}(\underline{\Omega}^{i}_{X})$ and $f_{*}\Omega^{i}_{\tX}$ are torsion free (by Lemma \ref{lemma:dBtorsionfree} in the first case, and obviously in the second). If $x \in X$ is a closed point, we have the exact sequence
$$ 0 \rightarrow \underline{\Gamma_{x}}(\cH^{0}(\Omega^{i}(K)))\rightarrow \cH^{1}(\bR \underline{\Gamma_{x}}(\underline{\Omega}^{i}_{X})) \rightarrow \cH^{1}(\bR \underline{\Gamma_{x}}(I\underline{\Omega}^{i}_{X})).$$
First, we claim
$$\underline{\Gamma_{x}}(\cH^{0}(\Omega^{i}(K))) \cong \cH^{1}(\bR \underline{\Gamma_{x}}(\underline{\Omega}^{i}_{X})) \quad \text{ for $i \leq k \leq n-2.$} $$
To show the claim, it suffices to show $\cH^{1}(\bR \underline{\Gamma_{x}}(I\underline{\Omega}^{i}_{X}))= 0$ for $i \leq k.$ By local duality,
$$\cH^{1}(\bR \underline{\Gamma_{x}}(I\underline{\Omega}^{i}_{X}))_{x} \cong Hom(\cH^{-1}( \bR \cH om_{X}(I\underline{\Omega}^{i}_{X}), \omega^{\bullet}_{X}))_{x}, k(x)).$$
Since $IC^{H}_{X}$ is a pure (polarizable) Hodge module of weight $n$, there is a quasi-isomorphism 
$$\bR \cH om_{X}(I\underline{\Omega}^{i}_{X}, \omega^{\bullet}_{X}) \cong I\underline{\Omega}^{n-i}_{X}[n].$$
So, we have an isomorphism
$$\cH^{1}(\bR \underline{\Gamma_{x}}(I\underline{\Omega}^{i}_{X})))_{x} \cong Hom(\cH^{n -1}(I\underline{\Omega}^{n-i}_{X}))_{x} , k(x)).$$
If $n -2 \geq k \geq i$, then $(n-i) + (n-1) >n$ and we must have 
\begin{equation}\label{eq:vanHDRIC}
\cH^{n-1}(I\underline{\Omega}^{n-i}_{X}))= 0. 
\end{equation}
This follows by the definition of the de Rham complex for a Hodge module, see \cite[\S 7]{schnell} for more details. So, we obtain
$$\cH^{1}(\bR \underline{\Gamma_{x}}(I\underline{\Omega}^{i}_{X}))) = 0 \quad \text{for $ i \leq k \leq n-2$.}$$ 
We know by Theorem \ref{HodgeThm} that we have $Gr^{0}_{F}H^{i-n}_{c}(X, DR(K)) = 0$. Also, by Proposition \ref{injProp}, $Gr^{W}_{i+1}H^{i-n}_{c}(X, DR(K)) = 0.$ So, we have 
$$ Gr^{i}_{F}H^{i-n}_{c}(X, DR(K)) =Gr^{i}_{F}Gr^{W}_{i}H^{i-n}_{c}(X, DR(K)) \cong Gr^{0}_{F} Gr^{W}_{i}H^{i-n}_{c}(X, DR(K))  =0.$$
Recall that we have an immersion $i_{X}:X \hookrightarrow \cX$ and, by the strictness of the Hodge filtration, we have
$$Gr^{i}_{F}H^{i-n}_{c}(X, DR(K))= H^{i-n}(\cX, Gr^{F}_{-i}DR(i_{X!}K)).$$
From the inclusion map $j: Y \hookrightarrow \cX$, we have
the maps
$$H^{i-n}_{c}(Y, Gr^{F}_{-i}DR(i_{+}K)) \rightarrow H^{i-n}(\cX, Gr^{F}_{-i}DR(i_{X!}K)) \rightarrow H^{i-n}(Y, Gr^{F}_{-i}DR(i_{+}K)).$$
The complex $Gr^{F}_{-i}DR(i_{+}K) = \Omega^{i}(K)[i-n]$ is well defined on $X$ and the previous diagram can simply be written as
$$H^{0}_{c}(X, \Omega^{i}(K)) \rightarrow H^{i-n}(\cX, Gr^{F}_{-i}DR(i_{X!}K)) \rightarrow H^{0}(X, \Omega^{i}(K)).$$
From \eqref{eq:DRtriangle}, we obtain
$$\cH^{j}(\Omega^{i}(K))= 0 \text{ if } j< 0.$$
Therefore, the spectral sequences, 
$$E_2 ^{ij}= H_?^i(X, \cH^j(Gr_{-i}^F DR(K))) \Rightarrow  H_?^{i+j}(X, Gr_{-i}^F DR(K))$$
collapse to isomorphisms 
$$H^{0}_{c}(X, \Omega^{i}(K)) \cong \Gamma_{c}(X, \cH^{0}(\Omega(K)))$$  
and
$$H^{0}(X, \Omega^{i}(K)) \cong \Gamma(X, \cH^{0}(\Omega^{i}(K))).$$
As a consequence, we have a commutative diagram 
$$ \bT \Gamma_{c}(X, \cH^{0}(\Omega^{i}(K))) \arrow[r] \arrow[hook, dr] & Gr^{i}_{F}H^{i-n}_{c}(X, DR(K)) \arrow[d] \\
& \Gamma(X, \cH^{0}(\Omega^{i}(K))). \eT$$
So the map
$$\Gamma_{c}(X, \cH^{0}(\Omega^{i}(K))) \rightarrow Gr^{i}_{F}H^{i-n}_{c}(X, DR(K))$$
must be injective, and we obtain
$$\Gamma(X, \underline{\Gamma_{x}}(\cH^{0}(\Omega^{i}(K)))) = \Gamma_{c}(X, \underline{\Gamma_{x}}(\cH^{0}(\Omega^{i}(K)))) \subseteq Gr^{i}_{F}H^{i-n}_{c}(X, DR(K))= 0.$$
Since $X$ is affine, we can conclude by the earlier claim that
$$\cH^{1}(\bR \underline{\Gamma_{x}}(\underline{\Omega}^{i}_{X})) \cong \underline{\Gamma_{x}}(\cH^{0}(\Omega^{i}(K))) = 0.$$

\end{proof}

\begin{cor}
     If $\tilde{S}$ is the the smallest closed subset of $X$ such that $X \backslash \tilde{S}$ has rational singularities, then $$\depth(\underline{\Omega}^{i}_{X}) \geq 2 \quad \text{for $i +2 \leq \depth_{\tilde{S}}(\cO_{X}).$}$$
\end{cor}

\begin{proof}
    Since $2 \leq \depth_{\tilde{S}}(\cO_{X})$, we have $X$ to be normal. We may apply the previous theorem with Corollary \ref{rat.rmk}.
\end{proof}

\begin{cor}
    If $\tilde{S}$ is the the smallest closed subset of $X$ such that $X \backslash \tilde{S}$ has rational singularities, and $n_{\tilde{S}} = \dim \tilde{S}$,
    $$\depth(\underline{\Omega}^{i}_{X}) \geq 2 \quad \text{for $i  + n_{\tilde{S}}+2 \leq \depth(\cO_{X}).$}$$
    In particular, Conjecture \ref{DepthConj.} holds when $X$ has rational singularities outside of a finite set.
\end{cor}

\begin{proof}
     This follows from the previous corollary and Lemma \ref{depthlemma}.
\end{proof}

\begin{rmk}
    Recently, a proof of Conjecture \ref{MuPoConj}, in the general case, has been worked out by Burke \cite{burke}.
\end{rmk}

\section{Differential forms and reflexive sheaves}

Before stating the main theorem of this section, we give a little background. We will also stay with the same notation as in the previous sections. Given an $n$-dimensional algebraic variety $X$, $S$ will be the singular locus with dimension $n_{S}$. We also let $U$ be the smooth locus with inclusion map $j: U \hookrightarrow X$.

For any variety $X$ and a resolution of singularities $\pi: \tilde{X} \rightarrow X$, we define
    $$\tilde{\Omega}^{i}_{X}:= \pi_{*}\Omega^{i}_{\tilde{X}}.$$
    If $\pi: \tilde{X} \rightarrow X$ is a strong log resolution, with expectional divisor $E$ with simple normal crossing support, we define
    $$\tilde{\Omega}^{i}_{X}(\log):= \pi_{*}\Omega^{i}_{\tilde{X}}(\log E).$$
    These sheaves are independent of the (strong log) resolution of singularities and there is always an injective map
    $$\tilde{\Omega}^{i}_{X} \rightarrow \tilde{\Omega}^{i}_{X}(\log).$$
    Since  Du Bois complexes are functorial, we get a natural map
\begin{equation}\label{eq:dB}
    \underline{\Omega}_X^p\to \R \pi_*\underline{\Omega}_{\tX}^p = \R \pi_*\Omega_{\tX}^p
\end{equation}
    and therefore, a canonical map
    $$\kappa:\cH^0(\underline{\Omega}_X^p)\to \cH^0( \R \pi_{*}\Omega_{\tX}^p)= \tilde \Omega_X^p.$$
    To unwind this a bit, given a simplicial resolution $\pi_{\bullet} :X_\bullet \to X$, we have
     $\underline{\Omega}_X^p= \R \pi_\bullet \Omega_{X_\bullet}^p$. This  is the total complex
    of the double complex representing
    $$ \R \pi_{0*}\Omega_{X_0}^p\to  \R \pi_{1*}\Omega_{X_0}^p\to \ldots$$
    Therefore, we  get a morphism
$$\underline{\Omega}_X^p= \R \pi_\bullet \Omega_{X_\bullet}^p\to \R \pi_{0*}\Omega_{X_0}^p$$
by projection. 
Since  one can build a simplicial resolution $\pi_{\bullet} :X_\bullet \to X$ with
    $X_0\to X$  equal to $\tilde X$ (cf \cite[\S 6.2]{deligneIII}), we obtain the morphism \eqref{eq:dB}.
We also have  a canonical map
$$\Omega_X^p\to \cH^0(\dB^p).$$
The composition of this  map with $\kappa$ is the map
$$\Omega_X^p\to \tilde \Omega_X^p$$
given by pullback of K\"ahler differentials. Therefore, we get the factorization given in the introduction. 

 

   \begin{lemma}
       If $X$ is normal, then for $0 \leq i \leq n$, $\cH^{1}(\bR \underline{\Gamma_{S}}(\underline{\Omega}^{i}_{X}))$ is a coherent.
   \end{lemma}

   \begin{proof}
       The distinguished triangle
       $$\bT \bR \underline{\Gamma_{S}}(\underline{\Omega}^{i}_{X}) \arrow[r] & \underline{\Omega}^{i}_{X} \arrow[r] & \bR j_{*}\Omega^{i}_{U} \arrow[r, "+1"] & \hfill \eT$$
       induces a the long exact sequence
       $$0 \rightarrow \cH^{0}( \underline{\Omega}^{i}_{X}) \rightarrow j_{*}\Omega^{i}_{U} \rightarrow \cH^{1}(\bR \underline{\Gamma_{S}}(\underline{\Omega}^{i}_{X})) \rightarrow \cH^{1}(\underline{\Omega}^{i}_{X}) \rightarrow \cdots$$
       Since $X$ is normal, $j_{*}\Omega^{i}_{U} \cong \cH om_{X}(\cH om_{X}(\Omega^{i}_{X}, \cO_{X}), \cO_{X})$ is coherent. We also know $\cH^{1}(\underline{\Omega}^{i}_{X})$ is coherent. So, $\cH^{1}(\bR \underline{\Gamma_{S}}(\underline{\Omega}^{i}_{X}))$ is a coherent by the long exact sequence.
   \end{proof}

\begin{thm}\label{DiffThm}
     Let $X$ be a normal variety and $f: \tX \rightarrow X$ a resolution of singularities. If $R^{i}f_{*}\cO_{\tX} = 0$ for $1 \leq i \leq k$, then 
     $$\depth_{S}(\underline{\Omega}^{i}_{X}) \geq 2 \quad \text{for $i \leq k.$}$$
    
\end{thm}

\begin{proof}
    We need to show $\cH^{1}(\bR \underline{\Gamma_{S}}(\underline{\Omega}^{i}_{X})) = 0$ for $i \leq k.$ For any closed point $x \in S$, we  have a spectral sequence
    $$E^{p,q}_{2} = \cH^{p}(\bR \underline{\Gamma_{x}}(\cH^{q}(\bR \underline{\Gamma_{S}}(\underline{\Omega}^{i}_{X})))) \rightarrow \cH^{p +q}(\bR \underline{\Gamma_{x}}\bR \underline{\Gamma_{S}}(\underline{\Omega}^{i}_{X})) \simeq \cH^{p+q}(\bR \underline{\Gamma_{x}}(\underline{\Omega}^{i}_{X})).$$
    By Theorem \ref{MuPoConj}, 
    $$\underline{\Gamma_{x}}(\cH^{1}(\bR \underline{\Gamma_{S}}(\underline{\Omega}^{i}_{X}))) \cong \cH^{1}(\bR \underline{\Gamma_{x}}(\underline{\Omega}^{i}_{X})) = 0.$$  
        So, if the dimension of the support of $\cH^{1}(\bR \underline{\Gamma_{S}}(\underline{\Omega}^{i}_{X}))$ is zero, we must have  $\cH^{1}(\bR \underline{\Gamma_{S}}(\underline{\Omega}^{i}_{X}))=0$.
    
    Since the statements are local, we can assume that $X$ is affine.
    We will use induction on the dimension of $X$ to prove that  the dimension of the support of $\cH^{1}(\bR \underline{\Gamma_{S}}(\underline{\Omega}^{i}_{X}))$ is zero. If $X$ is a surface, then $X$ has isolated singularities, and we may apply Theorem \ref{MuPoConj}. So consider when $X$ is a normal variety of dimension $n \geq 3$. By \cite[Prop. 4.18]{ks}, for a sufficiently general hyperplane $H$, we have
    $$\bR f_{*}\cO_{\tX} \otimes \cO_{H} \simeq \bR g_{*}\cO_{\tH},$$
    where $g: \tH \rightarrow H$ is a resolution of singularities. Furthermore, we have
    $$R^{i}f_{*}\cO_{\tX} \otimes \cO_{H} \cong R^{i}g_{*}\cO_{\tH}.$$
    So, we find that $H$ is normal and satisfies the statement's conditions. Note that we may choose $H$ such that the singular locus of $H$ is given by $S \cap H$, and therefore the codimension of $(S \cap H)$ in $H$ to be greater than or equal to $2$. 
    
    
    Let $U = X \backslash S$ with natural map $j: U \hookrightarrow X$, and $U_{H} = U \cap H = H \backslash (H \cap S)$ with natural map $p: U_{H} \hookrightarrow H.$ From the restriction map, we have the commutative diagram 
        $$\bT
   \bR p_{*}\underline{\Omega}^{i-1}_{U_H} \otimes \cO_{H}(-H) \arrow[r]  & \bR j_{*}\underline{\Omega}^{i}_{U} \otimes^{L} \cO_{H} \arrow[r]  & \bR p_{*}\underline{\Omega}^{i}_{U_H} \arrow[r, "+1"] & \hfill \\
   \underline{\Omega}^{i-1}_{H} \otimes \cO_{H}(-H) \arrow[r] \arrow[u]  & \underline{\Omega}^{i}_{X} \otimes^{L} \cO_{H} \arrow[r]  \arrow[u] & \underline{\Omega}^{i}_{H} \arrow[r, "+1"] \arrow[u] & \hfill.
    \eT$$
    The rows are exact by \cite[exp V, 2.2.1]{GNPP}.
    Hence, we have the exact triangle
    $$ \bT
    \bR \underline{\Gamma_{S \cap H}}(\underline{\Omega}^{i-1}_{H}) \otimes \cO_{H}(-H) \arrow[r]  & \bR \underline{\Gamma_{S }}(\underline{\Omega}^{i}_{X})\otimes^{L} \cO_{H} \arrow[r]  & \bR \underline{\Gamma_{S \cap H}}(\underline{\Omega}^{i}_{H}) \arrow[r, "+1"] \ & \hfill.
    \eT$$
    When we apply cohomology to the previous triangle, we have the exact sequence
    $$ \cdots \rightarrow \cH^{1}(  \bR  \underline{\Gamma_{S \cap H}}(\underline{\Omega}^{i-1}_{H})) \otimes \cO_{H}(-H) \rightarrow \cH^{1}(\bR \underline{\Gamma_{S }}(\underline{\Omega}^{i}_{X})\otimes^{L} \cO_{H}) \rightarrow \cH^{1}(\bR \underline{\Gamma_{S \cap H}}(\underline{\Omega}^{i}_{H})) \rightarrow \cdots. $$
    By the induction hypothesis, $\cH^{1}(  \bR  \underline{\Gamma_{S \cap H}}(\underline{\Omega}^{i-1}_{H})) \otimes \cO_{H}(-H) = \cH^{1}(\bR \underline{\Gamma_{S \cap H}}(\underline{\Omega}^{i}_{H}))= 0$. So, we obtain $ \cH^{1}(\bR \underline{\Gamma_{S }}(\underline{\Omega}^{i}_{X})\otimes^{L} \cO_{H})  = 0$ for $i \leq k$. The exact triangle
    $$\bT \bR \underline{\Gamma_{S}}(\underline{\Omega}^{i}_{X}) \otimes \cO_{X}(-H) \arrow[r] &  \bR \underline{\Gamma_{S }}(\underline{\Omega}^{i}_{X})\arrow[r] &\bR \underline{\Gamma_{S }}(\underline{\Omega}^{i}_{X})\otimes^{L} \cO_{H} \arrow[r,"+1"] & \hfill \eT$$
    induces the long exact sequence
    $$ \cdots \rightarrow \cH^{0}(\bR \underline{\Gamma_{S }}(\underline{\Omega}^{i}_{X})\otimes^{L} \cO_{H}) \rightarrow  \cH^{1}(\bR \underline{\Gamma_{S}}(\underline{\Omega}^{i}_{X})) \otimes \cO_{X}(-H) \rightarrow  \cH^{1}(\bR \underline{\Gamma_{S}}(\underline{\Omega}^{i}_{X})) \rightarrow \cH^{1}(\bR \underline{\Gamma_{S }}(\underline{\Omega}^{i}_{X})\otimes^{L} \cO_{H}) \rightarrow \cdots$$
    Thus, for $i \leq k$, there is a surjective map
    $$ \cH^{1}(\bR \underline{\Gamma_{S}}(\underline{\Omega}^{i}_{X})) \otimes \cO_{X}(-H) \rightarrow  \cH^{1}(\bR \underline{\Gamma_{S}}(\underline{\Omega}^{i}_{X})).$$
      Let $A = \supp{\cH^{1}(\bR \underline{\Gamma_{S }}(\underline{\Omega}^{i}_{X})})$ and suppose that $\dim A \geq 1.$ Then we may choose $H$ such that $A \cap H \neq \emptyset$. If we localize at any closed point $x \in A \cap H$, there is a surjective map
    $$ \cH^{1}(\bR \underline{\Gamma_{S}}(\underline{\Omega}^{i}_{X}))_x \otimes \cO_{X}(-H))_x \rightarrow  \cH^{1}(\bR \underline{\Gamma_{S}}(\underline{\Omega}^{i}_{X}))_x.$$
    Since $\cH^{1}(\bR \underline{\Gamma_{S}}(\underline{\Omega}^{i}_{X}))_{x}$ is a coherent sheaf, we can conclude that $\cH^{1}(\bR \underline{\Gamma_{S}}(\underline{\Omega}^{i}_{X}))_x = 0$ by applying Nakayama's lemma. Contradicting that $x \in A = \supp{\cH^{1}(\bR \underline{\Gamma_{S }}(\underline{\Omega}^{i}_{X})})$. Hence $\dim A =0$. \\

\end{proof}




\begin{cor}\label{cor:Omegareflex} 
     Let $X$ be a normal variety and $f: \tX \rightarrow X$ a resolution of singularities. If $R^{i}f_{*} \cO_{\tX} = 0$ for $1 \leq i \leq k$, then $\cH^{0}(\underline{\Omega}_{X}^{i})$  is reflexive and the natural maps
     $$ \cH^{0}(\underline{\Omega}_{X}^{i}) \rightarrow \tilde \Omega^{i}_{X} \rightarrow \tilde \Omega^{i}_{X}(\log) \rightarrow \Omega^{[i]}_{X}$$
     are isomorphisms for $i \leq k.$ In particular, if $\tilde{S}$ is the the smallest closed subset of $X$ such that $X \backslash \tilde{S}$ has rational singularities, then the morphisms above are isomorphisms for $i +2 \leq \depth_{\tilde{S}}(\cO_{X}).$
\end{cor}

\begin{proof}
     By Theorem \ref{DiffThm}, we have  $\cH^{1}(\bR \underline{\Gamma_{S}}(\underline{\Omega}^{i}_{X})) = 0$. So, the map  $\cH^{0}(\underline{\Omega}_{X}^{i}) \rightarrow \Omega^{[i]}_{X}:= j_{*}\Omega^{i}_{U}$ is an isomorphism. Since $X$ is normal, this isomorphism is equivalent to saying $\cH^{0}(\Om)$ is reflexive.  Since the sheaves $\tilde{\Omega}^{i}_{X}, \tilde{\Omega}^{i}_{X}(\log)$ are torsion free, and the natural maps $\cH^{0}(\underline{\Omega}_{X}^{i}) \rightarrow \tilde{\Omega}^{i}_{X}\rightarrow \tilde{\Omega}^{i}_{X}(\log)$ are isomorphisms over the smooth locus, they must also be isomorphisms. The last statement follows from Theorem \ref{DiffThm}and  Corollary \ref{rat.rmk}.
\end{proof}

\begin{rmk}
    If $\tilde{S}$ is the the smallest closed subset of $X$ such that $X \backslash \tilde{S}$ has rational singularities with dimension $n_{\tilde{S}},$ then by the previous corollary and Remark \ref{rat.rmk2}, we have the natural maps
     $$ \cH^{0}(\underline{\Omega}_{X}^{i}) \rightarrow \tilde \Omega^{i}_{X} \rightarrow \tilde \Omega^{i}_{X}(\log ) \rightarrow \Omega^{[i]}_{X}$$
     to be isomorphisms whenever $i \leq \depth(\cO_{X}) -n_{\tilde{S}} -2$. Therefore, the sheaves $ \Omega^{i}_{X}$ and $\tilde \Omega^{i}_{X}(\log)$ are reflexive for $i \leq \depth(\cO_{X}) -n_{\tilde{S}} -2$. This agrees with the results of Park \cite{park}, which give a better upper bound for $i$. A similar result was also given by Tighe \cite{tighe}.
\end{rmk}

    A recent topic of interest has been over higher singularities. When $X$ is a local complete intersection, the variety is said to have $k$-Du Bois singularities if the canonical morphism
$$\Omega^{p}_{X} \rightarrow \underline{\Omega}^{p}_{X}$$
is a quasi-isomorphism for $0 \leq p \leq k$. The variety $X$ is said to have $k$-rational singularities if the canonical morphism
$$\Omega^{p}_{X}\rightarrow \D(\underline{\Omega}^{n-p}_{X})[-\dim X]$$
is a quasi-isomorphism for $0 \leq p \leq k$, where $\D = \bR \cH om_{X}(\bullet, \omega_{X}^{\bullet}).$ See
the papers of  Musta\c{t}\u{a} et. al.  \cite{MuOlPoWi} and Jung et. al. \cite{JuKiSaYo} and Friedman-Laza \cite{FrLa} for more details.

For general quasi-projective varieties, we have the following definitions given by Shen, Vekatesh, and Vo \cite{shenVenVo}.

\begin{defn}
$X$ has $k$-Du Bois singularities if it is seminormal and
\begin{enumerate}
    \item codim$_{X}(S) \geq 2k+1$, where $S$ is the singular locus;
    \item $X$ has pre-$k$-Du Bois singularities. That is,
    $$\cH^{i}(\underline{\Omega}^{p}_{X}) = 0 \quad \text{for $i>0$ and $0 \leq p \leq k.$}$$
    \item $\cH^{0}(\underline{\Omega}^{p}_{X})$ is reflexive, for all $p \leq k.$
\end{enumerate}
\end{defn}

\begin{defn}
    $X$ is said to have $k$-rational singularities if it is normal and 
    \begin{enumerate}
    \item codim$_{X}(S) \geq 2k+2$, where $S$ is the singular locus;
    \item $X$ has pre-$k$-rational singularities. That is,
    $$\cH^{i}(\D(\underline{\Omega}^{n-p}_{X})[-\dim X])= 0 \quad \text{for $i>0$ and $0 \leq p \leq k.$}$$
\end{enumerate}
\end{defn}
Note that when $k = 0$, we receive the commonly known definitions for Du Bois and rational singularities, respectively. It was shown in \cite{shenVenVo} that all the respective definitions agree when $X$ is a local complete intersection. Furthermore, for normal varieties, the following implications were shown
$$\bT \text{$k$-rational singularities} \arrow[r, Rightarrow] \arrow[d, Rightarrow] & \text{$k$-Du Bois singularities} \arrow[d, Rightarrow] \\
\text{pre-$k$-rational singularities} \arrow[r, Rightarrow] & \text{pre-$k$-Du Bois singularities.} \eT$$
We end this section by noting that the condition of:
\begin{center}
    $\cH^{0}(\underline{\Omega}^{p}_{X})$ is reflexive, for all $p \leq k$
\end{center} 
for $k$-Du Bois singularities is satisfied if $X$ is normal and $R^{i}f_{*}\cO_{\tX} = 0$ for $1 \leq i \leq k$, where $f:\tX \rightarrow X$ is a resolution of singularities. This follows from the corollary \ref{cor:Omegareflex}.

\section{Criterion for isolated rational singularities}
In this section, we are interested in when the morphism
$$\cH^{0}(\Omp) \rightarrow \tilde{\Omega}^{p}_{X}$$
is an isomorphism. From the previous section, this morphism is an isomorphism for all $p$ when $X$ has rational singularities. We will show that the converse is false, but we can give a criterion for isolated singularities. For example, normal hypersurfaces of $\C^{n+1}$ whose defining equation are weighted-homogeneous
$$\displaystyle \sum_{i}x_{i}^{a_{i}} = 0 \quad \text{such that $\displaystyle \sum \dfrac{1}{a_{i}} \leq 1$}$$
never have rational singularities by \cite{Saito-Rat}. If we let $X$ denote such a hypersurface and $f: \tX \rightarrow X$ a resolution of singularities, then 
$$R^{i}f_{*}\cO_{\tX} = 0 \quad \text{for $1 \leq i \leq \dim X -2= n-2$}$$
$$R^{n-1}f_{*}\cO_{\tX} \neq 0.$$
Therefore, the natural map
$$\cH^{0}(\Omp) \rightarrow \tilde{\Omega}^{p}_{X}$$
is an isomorphism for $0 \leq p \leq n-2.$ In general, the natural map $\cH^{0}(\underline{\Omega}^{n-1}_{X}) \rightarrow \tilde{\Omega}^{n-1}_{X}$ is not an isomorphism. An example is when $X = \{z^3 + x^3 + y^3 = 0\}.$ However, there are cases when the map $\cH^{0}(\underline{\Omega}^{n-1}_{X}) \rightarrow \tilde{\Omega}^{n-1}_{X}$ is an isomorphism. An example is when $X = \{z^2 +x^3 +y^7 = 0\} \subset \C^{3}.$ These examples will be discussed in further detail.

To simplify notation, we give the following definition.

\begin{defn}
 We will say that $X$ has  { quasi-rational} singularities if it is normal and 
 $\cH^{0}(\underline{\Omega}_{X}^{p}) \rightarrow \tilde \Omega^{p}_{X}$  an isomorphism for all $p$. 
 
\end{defn}

 \begin{lemma}\label{lemma:quasirat}
  If $X$ is normal of dimension $n$, then
  $$\cH^{0}(\Om) \cong \tilde{\Omega}^{0}_{X}$$
   $$\cH^{0}(\underline{\Omega}^{n}_{X}) \cong \tilde{\Omega}^{n}_{X}$$
\end{lemma}

\begin{proof}
  The degree $0$ stated was proved in Lemma \ref{lemma:dBtorsionfree}.
  
  By \cite{GNPP}, there exists a simplicial resolution $\pi_\bullet: X_\bullet \to X$ with  $\dim X_i \le n-i$,
  and $X_0\to X$ a resolution.
  It follows from this and Grauert-Riemenschneider that
  $$\underline{\Omega}_X^n = \R \pi_{\bullet *}\Omega_{X_\bullet}^n =  \R \pi_{\bullet 0}\Omega_{X_\bullet}^n = \tilde \Omega_X^n$$

\end{proof}

\begin{thm}\label{thm:QR}
  Let $X$ be a normal variety with isolated singularities. Fix a resolution of singularities $\pi:\tilde X\to X$
  such that the exceptional divisor $E$ has simple normal crossings with components $E_i$.
 The morphism $\cH^{0}(\Omp) \rightarrow \tilde{\Omega}^{p}_{X}$ is an isomorphism if and only if the restriction maps $H^0(\tX, \Omega_{\tX}^p)\to H^0(E_i,\Omega_{E_i}^p)$
  are all zero.
\end{thm}

\begin{proof}
 Choose a Zariski open $U\subset X$ containing a unique singular point.
  Let $U_0 = \pi^{-1}U$ and let $U_1\to U_{0}\times_U U_0$
  be another resolution of singularities, with projection $\pi_i: U_1\to U_0$. Then
  $$\cH^0(\underline{\Omega}_X^p)(U) = \{\alpha\in \Omega^p(U_0)\mid \pi_1^* \alpha= \pi_2^* \alpha\}$$
  by \cite{huber}.
   Then $U_1$ has several components,  with one, say $U_{1,00}$ birational to $U$ and the
   rest  $U_{1,ij}$ birational to  $E_i\times E_j$. Suppose that  $H^0(\tX, \Omega_{\tX}^p)\to H^0(E_i,\Omega_{E_i}^p)$ are zero.
   If $\alpha\in \Omega^p(U_0)$, then
  $\pi_1^* \alpha= \pi_2^* \alpha$ holds on  $U_{1,00}$ since it birational to $U_0$. On other components, the equation
  $\pi_1^* \alpha= \pi_2^* \alpha$ is trivially true, because these forms vanish.
  Therefore
  $$\cH^0(\underline{\Omega}_X^p)(U) = \Omega^p_{U_0}(U_0)$$
Since this holds for all $U$, the morphism $\cH^{0}(\Omp) \rightarrow \tilde{\Omega}^{p}_{X}$ is an isomorphism.

  Suppose that
  $H^0(\tX, \Omega_{\tX}^p)\to H^0(E_i,\Omega_{E_i}^p)$ is nonzero for some $i$. Let $\alpha$ be a nonzero $p$-form in the image.
  Then the condition $\pi_1^* \alpha= \pi_2^* \alpha$ will fail on  $U_{1,ii}$.
  Therefore
  $$\cH^0(\underline{\Omega}_X^p)(U) \subsetneq \Omega^p_{U_0}(U_0)$$
  
\end{proof}

\begin{cor}
  A normal projective surface $X$  has quasi-rational singularities if and only if all exceptional
  curves in a resolution $\tX$ have trivial image in the Albanese $Alb(\tX)$.
  In particular, $X$ has quasi-rational singularities if all irreducible exceptional curves are rational.
\end{cor}

\begin{proof}
 By  Lemma \ref{lemma:quasirat},  it is enough to check that the map $\cH^{0}(\underline{\Omega}^{1}_{X}) \rightarrow \tilde{\Omega}^{1}_{X}$ is an isomorphism. Now observe that 
 $H^0(\tX, \Omega_{\tX}^1)\to H^0(E_i,\Omega_{E_i}^1)$ is zero if and only if $E_i$ maps trivially to  $Alb(\tX)$.
\end{proof}

\begin{ex}\label{ex:cusp}
Recall that a normal surface singularity is a cusp if the exceptional divisor  $E$ in the minimal resolution is a cycle of smooth rational curves.
Such a singularity is quasi-rational but not rational because $H^1(E, \cO_E)\not=0$.
\end{ex}

\begin{prop}
  Let $X$ be a quasiprojective variety such that
  $\cH^{0}(\Omp) \rightarrow \tilde{\Omega}^{p}_{X}$ is an isomorphism.
  For a general hyperplane section $H$, the morphism  $\cH^{0}(\underline{\Omega}^{p-1}_{H}) \rightarrow \tilde{\Omega}^{p-1}_{H}$ is an isomorphism. In particular, if $X$ has quasi-rational singularities, so does $H$.
\end{prop}

\begin{proof}
   Choose a very ample line bundle $L$ on $X.$ Let $H$ be a general divisor associated to $L$. Then $H$ is normal, and the strict transform $\tilde{H}$ under $f: \tX \rightarrow X$ is nonsingular by Bertini's theorem.  Pushing forward the standard exact sequence
    $$0 \rightarrow \Omega_{\tilde{H}}^{p-1} \otimes f^{*}L^{-1} \rightarrow \Omega^{p}_{\tX}\vert_{\tilde{H}} \rightarrow \Omega^{p}_{\tilde{H}} \rightarrow 0$$
    yields an exact sequence
    $$0 \rightarrow \tilde{\Omega}_{H}^{p-1} \otimes L^{-1} \rightarrow \tilde{\Omega}^{p}_{X}\vert_{H} \rightarrow \tilde{\Omega}^{p}_{H}.$$
    On the other hand, by \cite{GNPP}, we have a distinguished triangle
    $$\bT\underline{\Omega}^{p-1}_{H} \otimes L^{-1} \ar[r] &\underline{\Omega}^{p}_{X}\vert_{H} \ar[r]& \underline{\Omega}^{p}_{H} \ar[r, "+1"] & \hfill. \eT$$
    This triangle is compatible with the previous exact sequence in the sense that we have a commutative diagram
    $$ \bT  0 \ar[r] &\cH^0( \underline{\Omega}^{p-1}_{H}) \otimes L^{-1} \ar[r] \ar[d] & \cH^0(\underline{\Omega}^{p}_{X}\vert_{H}) \ar[r] \ar[d] & \cH^0(\underline{\Omega}^{p}_{H}) \ar[d] \\
    0 \ar[r] & \tilde{\Omega}_{H}^{p-1} \otimes L^{-1} \ar[r] & \tilde{\Omega}^{p}_{X}\vert_{H} \ar[r] &\tilde{\Omega}^{p}_{H}. \eT$$
    Since the middle vertical arrow is an isomorphism, and the last arrow is injective, the first vertical arrow is an isomorphism. Therefore
    $$\cH^{0}(\underline{\Omega}^{p-1}_{H}) \rightarrow \tilde{\Omega}^{p-1}_{H}$$
    is an isomorphism.
\end{proof}

\begin{prop}
  Let $X$ be normal projective variety with quasi-rational singularities, and $f:\tX\to X$ a resolution of singularities. Then
  \begin{equation}
    \label{eq:HpO}
  H^p(X,\cO_X)\to H^p(\tilde X,\cO_{\tX})   
  \end{equation}
  is surjective for all $p$. This is an isomorphism if $p=1$, or if $X$ is Du Bois and $Gr_F^0Gr^W_iH^p(X,\C)=0$
  for $i<p$.
\end{prop}

\begin{proof}
  We claim that  $F^p\cap W_{p-1}H^p(X,\C)=0$.
Let $S\subset X$ be the singular locus, and $E= f^{-1}S$.   By induction on dimension, we
can assume that $F^p\cap W_{p-1}H^p(S,\C)=0$.
We have the standard ``Mayer-Vietoris" exact sequence
   $$H^{p-1}(E, \C) \to H^p(X,\C)\to H^{p}(\tilde X,\C)\oplus H^{p}(S,\C)\to
   H^p(E,\C)\ldots $$
 Since $F^pH^{p-1}(E,\C)=0$, we see from this and that the strictness properties of $F$ and $W$ that $F^p\cap W_{p-1} H^p(X,\C)$ injects into
$$ F^p\cap W_{p-1}H^{p}(\tilde X,\C)\oplus F^p\cap W_{p-1}H^{p}(S,\C) =0$$
This proves the claim.

    We can identify
     $$H^0(X,\tilde \Omega_X^p)  = F^pH^p(\tilde X,\C)$$
     and
     $$H^0(X,\underline{\Omega}_X^p)\cong  F^pH^p(X, \C)\cong F^pGr^W_{p}H^p(X, \C).$$
     The last isomorphism follows from the above claim and by observing that $W_pH^p(X, \C)= H^p(X, \C)$ because $X$ is projective. The map
     \begin{equation}
       \label{eq:qROm}
      H^0(X,\underline{\Omega}_X^p)\to H^0(X,\tilde \Omega_X^p) 
     \end{equation}
           is an isomorphism because $X$ has quasi-rational singularities. Therefore
     $$Gr_F^0Gr^W_pH^p(X,\C)\to Gr_F^0H^p(\tX,\C)$$
     is an isomorphism because it is complex conjugate to the previous map.
     The surjectivity of  \eqref{eq:HpO} follows from the commutative diagram
$$
\xymatrix{
 Gr_F^0Gr^W_pH^p(X,\C)\ar[r] & Gr_F^0H^p(\tX,\C) \\
 H^p(X,\cO_X)\ar[r]\ar[u] & H^p(\tilde X,\cO_{\tX})\ar[u]^{=}
}
$$

When $p=1$, the  map \eqref{eq:HpO}
is injective because it is an edge map for the Leray spectral sequence. Therefore, we have an isomorphism in this case.
When $X$ is Du Bois, we can identify the \eqref{eq:HpO}  with
$$Gr_F^0H^p(X,\C)\to Gr_F^0H^p(\tilde X,\C)$$
The kernel of this map is $Gr_F^0 W_{p-1}H^p(X,\C)$. This would vanish if $Gr_F^0Gr^W_iH^p(X,\C)=0$
  for all $i<p$. Therefore, \eqref{eq:HpO} is injective with this assumption.

\end{proof}

For the following statement, it is worth observing that by excision, the mixed Hodge structure $H_x^p(X,\Z)$
is an invariant of the singularity at $x$.

\begin{thm}
  Suppose that $X$ is a normal variety with isolated singularities. Then
  $X$ has rational singularities if and only if the following conditions hold:
  \begin{enumerate}
  \item[(R1)] $X$ has quasi-rational singularities.
  \item[(R2)] $X$ is Du Bois.
    \item[(R3)] For all $x\in X$, all $0<p\le \dim X$, and all $k<p$, $Gr_F^0Gr^W_kH_x^p(X,\C)=0$.
  \end{enumerate}
  
\end{thm}

\begin{proof}
  Suppose $X$ has rational singularities. Then $X$ has quasi rational singularities by Corollary \ref{cor:Omegareflex} 
  and it has Du Bois singularities by \cite{kovacs}. We check condition (R3). Let  $0<p\le \dim X$.
  Choose a log resolution $\pi:\tX\to X$,
  a contractible neighbourhood $x\in B\subset X$ of a singular point. Let $\tilde B= \pi^{-1}B$ and $E\subset \tilde B$ the reduced exceptional divisor. By excision $H^p_x(X) = H_x^p(B) = H^p_E(\tilde B)$. Therefore, we have an exact sequence
  $$H^{p-1}(\tilde B)\to H^{p-1}(B-x)\to H_x^p(X)\to H^p(\tilde B) = H^p(E)$$
  When $0<p\le \dim X$, a standard consequence of the decomposition theorem is that 
  the first map is surjective (cf  \cite[thm 1.11]{SteenbrinkIsolated}). Therefore,
   we have an injection
  $$H^p_x(X)\to H^p(E)$$
  This map is compatible with mixed Hodge structures. Therefore
  $$Gr^0_FH^p_x(X)\to Gr^0_FH^p(E)= H^p(E, \cO_E)$$
  is injective. The second space is zero because $(R^p\pi_*\cO_{\tX})_x$ surjects onto it.

  Next, we prove the converse. Suppose that (R1), (R2), and (R3) hold.
  Since the conditions are local,  we can assume that $X$ is projective.
  Let $U\subset X$ denote the smooth locus, and $S= X-U$.
Since $U$ is smooth, the mixed Hodge structure on $H^p(U)$ has weights $\ge p$.
  Therefore, the  exact sequence
  $$ \bigoplus_{x\in S} H_x^p(X, \C)\to H^p(X,\C)\to H^p(U, \C)$$
  implies that the map
  $$\bigoplus Gr^0_F Gr^W_k H^p_x(X)\to Gr^0_F Gr^W_k H^p(X)$$
  is surjective when $k<p$.
  Therefore, we have isomorphisms
  $$H^p(X,\cO_X)\to H^p(\tilde X,\cO_{\tX})$$
  by the previous proposition.

  The  sheaves $R^{i}\pi_{*}\cO_{\tX}$ have zero dimensional support. Therefore, the Leray spectral sequence
    $$E^{p,q}_{2}= H^{p}(X, R^{q}\pi_{*}\cO_{\tX})$$
    reduces to a long, exact sequence
    
\begin{equation}\label{eq:LerayExS}
    \bT \cdots \ar[r] & E^{p,0}_{2} \ar[r, "{\sim}"]& H^{p}(\tX, \cO_{\tX}) \ar[r] & E^{0,p}_{2} \ar[r] & E^{p+1,0} \ar[r, "{\sim}"] &\cdots. \eT 
\end{equation}

Consequently $H^0(X,R^{i}\pi_{*}\cO_{\tX}) = 0$ for $i >0$, which implies that the sheaves are zero.

\end{proof}

We can see by the following examples that none of the conditions in the above theorem can be omitted.

\begin{ex}
 A cusp singularity   satisfies (R1) (Example \ref{ex:cusp}) and  (R2) \cite[thm 3.8]{SteenbrinkIsolated} but not (R3).
\end{ex}

\begin{ex}
 The  singularity $X$ defined by  $ z^2+x^3+y^7=0$ satisfies
 (R1) and (R3) but not (R2). Indeed, for (R1) and (R3), it suffices to observe that the exceptional divisor in a certain resolution is a chain of smooth rational curves (see \cite[pp 197-198]{lipmanRes}). However, it is not Du Bois by \cite[thm 3.8]{SteenbrinkIsolated}.
\end{ex}

\begin{ex}
 A cone over an elliptic curve satisfies (R2)  \cite[thm 3.8]{SteenbrinkIsolated} and (R3) but not (R1).
\end{ex}

\printbibliography 

\end{document}